\documentclass[10pt]{article}
\usepackage{times}

\DeclareMathVersion{varnormal}
\DeclareMathVersion{varbold}

\SetSymbolFont{letters}{normal}{OML}{txmi}{m}{it}
\SetSymbolFont{letters}{bold}{OML}{txmi}{bx}{it}
\SetSymbolFont{letters}{varnormal}{OML}{mdugm}{m}{it}
\SetSymbolFont{letters}{varbold}{OML}{mdugm}{b}{it}
\SetSymbolFont{operators}{normal}{OT1}{txr}{m}{n}
\SetSymbolFont{operators}{bold}{OT1}{txr}{bx}{n}
\SetSymbolFont{operators}{varnormal}{OT1}{mdugm}{m}{n}
\SetSymbolFont{operators}{varbold}{OT1}{mdugm}{b}{n}
\SetSymbolFont{symbols}{normal}{OMS}{txsy}{m}{n}
\SetSymbolFont{symbols}{bold}{OMS}{txsy}{bx}{n}
\SetSymbolFont{symbols}{varnormal}{OMS}{mdugm}{m}{n}
\SetSymbolFont{symbols}{varbold}{OMS}{mdugm}{b}{n}
\SetSymbolFont{largesymbols}{normal}{OMX}{txex}{m}{n}
\SetSymbolFont{largesymbols}{bold}{OMX}{txex}{bx}{n}
\SetSymbolFont{largesymbols}{varnormal}{OMX}{mdugm}{m}{n}
\SetSymbolFont{largesymbols}{varbold}{OMX}{mdugm}{b}{n}

\SetMathAlphabet{\mathrm}{normal}{OT1}{txr}{m}{n}
\SetMathAlphabet{\mathrm}{bold}{OT1}{txr}{bx}{n}
\SetMathAlphabet{\mathrm}{varnormal}{OT1}{mdugm}{m}{n}
\SetMathAlphabet{\mathrm}{varbold}{OT1}{mdugm}{b}{n}

\SetMathAlphabet{\mathit}{normal}{OT1}{txr}{m}{it}
\SetMathAlphabet{\mathit}{bold}{OT1}{txr}{bx}{it}
\SetMathAlphabet{\mathit}{varnormal}{OT1}{mdugm}{m}{it}
\SetMathAlphabet{\mathit}{varbold}{OT1}{mdugm}{b}{it}

\usepackage{amsmath}
\usepackage{amsthm}
\usepackage{mathrsfs}
\usepackage{amssymb}
\usepackage{changepage}   
\usepackage[ruled]{algorithm2e}
\usepackage{program}
\usepackage{empheq}
\usepackage{subcaption}
\usepackage[paper=a4paper,top=100pt,left=100pt,right=100pt,bottom=100pt,foot=50pt]{geometry}
\usepackage{titlesec}
\usepackage{lineno}
\usepackage{hyperref}
\usepackage{nicefrac}
\usepackage{float}
\usepackage{upgreek}
\usepackage[dvipsnames,table]{xcolor}
\usepackage{verbatim}
\usepackage{listings}
\usepackage{mdframed}
\usepackage[T1]{fontenc}

\titleformat{\section}{\normalfont\bfseries}{\thesection}{1em}{}
\titleformat{\subsection}{\normalfont}{\thesubsection}{1em}{}
\titleformat{\subsubsection}{\normalfont\it}{\thesubsubsection}{1em}{}

\renewcommand\appendix{\par
  \setcounter{section}{0}
  \setcounter{subsection}{0}
  \setcounter{figure}{0}
  \setcounter{table}{0}
  \setcounter{equation}{0}
  \setcounter{algocf}{0} 
  \setcounter{lstlisting}{0}
  \renewcommand\thesection{Appendix \Alph{section}}
  \renewcommand\thefigure{\Alph{section}.\arabic{figure}}
  \renewcommand\thetable{\Alph{section}.\arabic{table}}
  \renewcommand\theequation{\Alph{section}.\arabic{equation}}
  \renewcommand{\thealgocf}{\Alph{section}.\arabic{algocf}}
  \renewcommand{\thelstlisting}{\Alph{section}.\arabic{lstlisting}}
  \renewcommand\theHsection{Appendix \Alph{section}}
  \renewcommand{\theHlstlisting}{\theHsection.\arabic{lstlisting}}
}

\definecolor{cmtcolor}{RGB}{128,128,128}
\definecolor{kwdcolor}{RGB}{180,20,128}
\definecolor{strcolor}{RGB}{0,120,80}

\newcommand\digitstyle{\color{kwdcolor2}}
\makeatletter
\newcommand{\ProcessDigit}[1]
{%
  \ifnum\lst@mode=\lst@Pmode\relax%
   {\digitstyle #1}%
  \else
    #1%
  \fi
}

\definecolor{hltcolor}{RGB}{235,120,60}

\theoremstyle{plain}
\newtheorem{theorem}{Theorem}[section]
\newtheorem{lemma}[theorem]{Lemma}
\newtheorem{corollary}[theorem]{Corollary}
\newtheorem{thmRule}[theorem]{Rule}
\newtheorem{definition}[theorem]{Definition}

\newtheoremstyle{myremark}
    {\dimexpr\topsep/2\relax} 
    {\dimexpr\topsep/2\relax} 
    {}          
    {}          
    {\bfseries\itshape} 
    {.}         
    {.5em}      
    {}          
\theoremstyle{myremark}
\newtheorem{remark}{Remark}

\def\ttt#1;{\texttt{#1}} 

\def\bgnEqn{\begin{equation}}
\def\endEqn{\end{equation}}

\title{\large Dekker's floating point number system and compensated summation algorithms}

\author
{
  \small Longfei Gao\thanks{Email address: longfei.gao@anl.gov} \\
  {
    \it \small Argonne National Laboratory 
  } \\
  \small Frimpong Baidoo\thanks{Email address: fabaidoo@utexas.edu} \\
  {
    \it \small Oden Institute for Computational Sciences and Engineering, University of Texas at Austin
  }
}

\date{}

\begin{document}

\maketitle

\begin{abstract}
The recent hardware trend towards reduced precision computing has reignited the interest in numerical techniques that can be used to enhance the accuracy of floating point operations beyond what is natively supported for basic arithmetic operations on the hardware.
In this work, we study the behavior of various compensated summation techniques, which can be used to enhance the accuracy for the summation operation, particularly in situations when the addends are not known \textit{a priori}.
Complete descriptions of the error behavior are provided for these techniques. 
In particular, the relationship between the intermediate results at two consecutive summing steps is provided, which is used to identify the operation that limits accuracy and guide the design of more nuanced techniques.
The analysis relies on the work of Dekker \cite{dekker1971floating}, which uses a special floating point number system that does not require uniqueness in the number representation.
Despite this seemingly strange attribute, Dekker's system is very convenient for the analysis here and 
allows general statements to be expressed succinctly compared to a number system that requires uniqueness.
To prepare the foundation for discussion, we start by giving a thorough exhibition of Dekker's number system and supply the details that were omitted in \cite{dekker1971floating}. 
Numerical examples are designed to explain the inner workings of these compensated summation techniques, illustrate their efficacy, and empirically verify the analytical results derived.
Discussions are also given on application scenarios where these techniques can be beneficial. 
\end{abstract}

\section{Introduction}
In recent years, because of the drastically increasing demand from artificial intelligence (AI) and machine learning (ML) applications, the computing 
industry has gravitated towards data formats narrower than the IEEE double precision format that most computational scientists and engineers are accustomed to.
As a result, there has been a renewed interest in techniques that can enhance the accuracy of numerical operations. 
Iterative refinement \cite{carson2017new,carson2018accelerating,haidar2018harnessing} and matrix multiplications with multi-word arithmetic \cite{henry2019leveraging,fasi2023matrix,ootomo2024dgemm} are among the most notable examples of such techniques.
%

In this work, we focus on the compensated summation algorithms, which can be viewed as techniques that enhance the accuracy of the summation operation, but can also be applied more flexibly in situations such as ODE or PDE simulations. 
These techniques date back to at least the 60s \cite{kahan1965pracniques,moller1965quasi}.
%
The building blocks for the compensated summation algorithms are error free transformations (EFTs) over addition.
EFTs have often been attributed to Dekker's work \cite{dekker1971floating}, which describes multi-word arithmetic for basic floating point operations.
Although Dekker did not coin the term EFT in \cite{dekker1971floating}, the concept 
has been broadly adopted and found success in many areas \cite{quinlan1994round,shewchuk1997adaptive,chowdhary2022fast}.

Two EFTs over addition were given in \cite{dekker1971floating}, referred to hereafter as the 3op version and the 6op version, respectively, named after how many floating point operations are involved. 
The 3op version requires certain ordering on the two addends to work,
whereas the 6op version works for all number combinations in Dekker's floating point number system. 
Dekker's work \cite{dekker1971floating} employs a unique floating point number system that does not require uniqueness in the number representation, which turns out to be very convenient for the analysis and allows more insights.
Moreover, the results derived from Dekker's system can be transferred to the more familiar IEEE system with minimal effort since both systems address the same set of numbers.

Despite being instrumental, Dekker's work \cite{dekker1971floating} lacks the detail with several key results stated without proof, e.g., 
Lemma \ref{lemma:relationBetweenFloatingPointRoundingAndMantissaIntegerRounding} 
and
Theorem \ref{thm:6opEFTadditionDekker} below. 
For this reason, we revisit Dekker's system in Section \ref{sec:FloatingPointNumberSystem} of this work, supply the missing details, and provide rigorous proofs for a few useful results regarding rounding and addition. 
The comparison between Dekker's system and the more familiar IEEE system is also given, along with their links.
Detailed proofs for the 3op and 6op EFTs over addition stated in Dekker's system are then given in Section \ref{sec:EFTunderAddition}, again, filling in the missing details from \cite{dekker1971floating} and covering the various corner cases. 
Moreover, subnormal numbers are natually included in the analysis here, which are often neglected in earlier works such as \cite{moller1965quasi,linnainmaa1974analysis,knuth1997art_vol2}. 
These details are becoming increasingly important as adaptations or new designs of algorithms are being called upon to accommodate 
emerging hardware that may feature non-traditional data formats,
special rounding modes, 
and mixed-precision instructions.

The next portion of this work applies the derived EFT results to the analysis of compensated summation algorithms. 
The classical error analysis from \cite{knuth1997art_vol2} and \cite{goldberg1991every} is rather cumbersome and error-prone, which also tracks only the leading order term and addresses only the 3op-compensated recursive summation. 
Its extension to the 6op-compensated recursive summation is extremely difficult.
Instead, we provide a cleaner analysis that makes use of the EFT property and give a complete description of the error behavior of the 6op-compensated recursive summation algorithm in Section \ref{sec:sumSequence6op}. 
One particularly appealing aspect of our analysis is the derivation of the relation between the intermediate results at two neighboring steps, which reveals the pattern of error propagation and identifies the operation that limits the accuracy of the overall algorithm. 
More nuanced compensated summation algorithms are then given in Sections \ref{sec:sumSequenceDouble6op} and \ref{sec:sumSequenceTriple6op}, which remove this bottleneck.
These more nuanced versions are also analyzed using the techniques developed here.
Full descriptions of their error behaviors are provided, including the error propagation patterns.
Aside from the improved accuracy, the small error bounds equipped with these advanced versions of compensated summation algorithms can also be leveraged in computer system validation tasks as explained in Section \ref{example:accumulation}.

Numerical examples are provided in the final portion of this work.
A pedagogical example is designed to help the readers understand the inner workings of these compensated summation algorithms as well as how their behaviors differ in certain situations.
%
A second example compares the accuracy of the various recursive summation algorithms, illustrating the benefit of compensation and verifying empirically that the advanced versions can achieve an accuracy as if a data format with twice the amount of mantissa bits has been used in the regular recursive summation.
%
%
A third example is given in the context of dynamical system simulation using a classical three-body configuration that admits a figure-eight shaped periodic orbit. 
It is shown that the compensated algorithms significantly improve the orbit stability in simulations.
%

The remainder of this work is outlined as follows: In Section \ref{sec:FloatingPointNumberSystem}, we describe and analyze Dekker's floating point number system; In Section \ref{sec:EFTunderAddition}, we give detailed proof for the 3op and 6op EFTs over addition; In Section \ref{sec:compensatedSum}, we analyze the various compensated summation algorithms; In Section \ref{sec:numericalExamples}, numerical examples are provided; In Section \ref{sec:conclusions}, conclusions are drawn.

\section{Floating point number system}
\label{sec:FloatingPointNumberSystem}

In Dekker's work \cite{dekker1971floating}, the author presented a few techniques for error free transformation and multi-word arithmetic.
%
Dekker used a somewhat special floating point number system that does not require uniqueness in the number representation, which turns out to be quite convenient in proving results and making statements more general.
In the following, we first take a close look at Dekker's system, which we rely on heavily in this work. 

\subsection{Numbers and representations}
Elements of Dekker's floating point number system have the following form:
\bgnEqn
\label{eqn:DekkerNumberRepresentation}
x = m(x) \times \beta^{e(x)},
\endEqn
where $m(x) \in \mathbb Z$ 
denotes the mantissa of $x$
and $e(x) \in \mathbb Z$ 
denotes the exponent of $x$, with $\mathbb Z$ being the set of integers.
Dekker's floating point number system is defined as 
\bgnEqn
\label{eqn:DekkerFloatingPointSystem}
\left\{ x \,\vert\, x = m(x) \times \beta^{e(x)},\, \vert m(x)\vert < M,\, E_{\min} \leq e(x) \leq E_{\max} \right\},
\endEqn
where $\beta$, $M$, $E_{\min}$, and $E_{\max}$ are integers with $\beta>1$ and $M>0$. 
Denoting the number of mantissa digits as $t$, we assume $M = \beta^t$ below, which is natural.

Our primary interest is in the binary floating point number system, i.e., $\beta=2$. 
However, we will keep the discussion general with respect to the base $\beta$ and make it explicit when we specialize to the case $\beta = 2$.
Even when we specialize to the case $\beta=2$, we may still use $\beta$ to denote the base because it is immediately apparent that we are referring to the base of the system, whereas the number $2$ may come from other contexts.

Before proceeding, we comment on the subtle distinction between a {\it number} and a {\it representation} of a number here. 
A number is a unique entity in the number system, which can have multiple representations, e.g., $\nicefrac{1}{2}$, $0.5$, and $5 \times 10^{-1}$ can all be used to represent the same number in the real number system.
Numbers can be compared to one another. 
We may also compare two representations, which really means comparing the underlying numbers that they represent. Specifically, all representations of the same number compare equal to one another.

In Dekker's system, a floating point number is allowed to have multiple representations.
To familiarize ourselves with Dekker's system, we may group the representations in it into different {\it bins} based on their exponent. 
For example, with $\beta=2$, the bin associated with an arbitrary exponent $E$, where $E_{\min} \leq E \leq E_{\max}$, is illustrated in \eqref{eqn:representationsWithExponentE-oneGroup}.
To simplify the discussion, we include only the positive numbers in \eqref{eqn:representationsWithExponentE-oneGroup} and similar illustrations because Dekker's system is symmetric.
Mantissas of these representations vary from $t$ digits of $0$'s to $t$ digits of $1$'s, which correspond to the smallest number $0$ and the largest number $(\beta^t-1) \times \beta^{E}$ representable in this bin, respectively.
%
%
\begin{equation}
\label{eqn:representationsWithExponentE-oneGroup}
\underbrace{0 \cdots 0}_{\text{$t$ digits}} \times \,\beta^{E}
\longleftrightarrow
\underbrace{1 \cdots 1}_{\text{$t$ digits}} \times \,\beta^{E}
\end{equation}
%

If $t>1$, representations in this bin can be divided into two groups as illustrated in \eqref{eqn:representationsWithExponentE-twoSubgroups}. 
The group of representations in 
\eqref{eqn:representationsWithExponentE-twoSubgroups:firstDigit=1} has the leading digit in mantissa taking value $1$ whereas 
the group 
in 
\eqref{eqn:representationsWithExponentE-twoSubgroups:firstDigit=0} has the leading digit in mantissa taking value $0$.
Representations with a non-zero leading digit in mantissa, e.g., those in 
\eqref{eqn:representationsWithExponentE-twoSubgroups:firstDigit=1}, are referred to as \textit{normal} whereas 
representations with a zero leading digit in mantissa, e.g., those in 
\eqref{eqn:representationsWithExponentE-twoSubgroups:firstDigit=0}, are referred to as \textit{subnormal}. 
With $\beta=2$, there are equal amount of normal and subnormal representations in each bin.
\begin{subequations}
\label{eqn:representationsWithExponentE-twoSubgroups}
\begin{align}
1\underbrace{0 \cdots 0}_{\text{$t-1$}} \times \,\beta^{E}
\longleftrightarrow
1\underbrace{1 \cdots 1}_{\text{$t-1$}} \times \,\beta^{E}
\label{eqn:representationsWithExponentE-twoSubgroups:firstDigit=1}
\\
0\underbrace{0 \cdots 0}_{\text{$t-1$}} \times \,\beta^{E}
\longleftrightarrow
0\underbrace{1 \cdots 1}_{\text{$t-1$}} \times \,\beta^{E}
\label{eqn:representationsWithExponentE-twoSubgroups:firstDigit=0}
\end{align}
\end{subequations}

Assuming $E-1 \geq E_{\min}$, representations in 
\eqref{eqn:representationsWithExponentE-twoSubgroups:firstDigit=0} admit equivalent representations as illustrated in \eqref{eqn:representationsWithExponentE:SubgroupFirstDigit=0:RepresentedWith(E-1)}, where all representations have one \textit{trailing} zero with the first $t-1$ digits of their mantissas vary from all $0$'s to all $1$'s, all with exponent $E-1$. 
\begin{equation}
\label{eqn:representationsWithExponentE:SubgroupFirstDigit=0:RepresentedWith(E-1)}
\underbrace{0 \cdots 0}_{\text{$t-1$}}0 \times \beta^{E-1}
\longleftrightarrow
\underbrace{1 \cdots 1}_{\text{$t-1$}}0 \times \beta^{E-1}
\end{equation}
On the other hand, representations in 
\eqref{eqn:representationsWithExponentE-twoSubgroups:firstDigit=1} do not admit equivalent representations with exponent $E-1$ in Dekker's system because of the constraint $\vert m(x) \vert < M = \beta^t$ from \eqref{eqn:DekkerFloatingPointSystem}.

Imagine a process in which we list all representations in the system from bin $E_{\min}$ to bin $E_{\max}$.
Based on the observation above, we have that with the exception of the first bin in this process where $E = E_{\min}$, 
for each new bin introduced subsequently, the normal representations are new numbers added to the system whereas the subnormal representations have already appeared in previously introduced bins.

Notice also that the group of representations from \eqref{eqn:representationsWithExponentE-twoSubgroups:firstDigit=0} can be further divided into two subgroups based on if the second leading digit is also zero. 
The subgroup with two \textit{leading} zeros also admits equivalent representations in bin $E-2$, assuming $E-2 \geq E_{\min}$, each with two \textit{trailing} zeros.
The overlapping pattern between different bins should be clear now and can be stated generally as follows:
Given $k \in \mathbb Z$, if $1 \leq k \leq t$ and $E-k \geq E_{\min}$, those representations in \eqref{eqn:representationsWithExponentE-oneGroup} with $k$ leading zeros admit equivalent representations with exponent $E-k$. 
The fraction of these representations in the entire $\beta^t$ representations of \eqref{eqn:representationsWithExponentE-oneGroup} is $\tfrac{1}{\beta^k}$.
In the case $k=t$, there is only one overlapping number, zero, between the two bins with exponents $E$ and $E-k$.
When $k>t$, there is no overlapping number between the two bins.

Based on the observation above, 
we have the following lemma on the uniqueness of normal representation.
\begin{lemma}[Normal representation is unique]
\label{lemma:NormalRepresentationIsUnique} 
A number from Dekker's system can admit at most one normal representation.\\
\normalfont
{\bf Proof.}
Suppose $m(a) \times \beta^{e(a)}$ and $m(b) \times \beta^{e(b)}$ are two distinct normal representations.
If $e(a) = e(b)$ but $m(a) \neq m(b)$, it is obvious that these two representations cannot represent the same number.

If $e(a) \neq e(b)$, without loss of generality, assume $m(a)>0$, $m(b)>0$, and $e(a) < e(b)$. 
On one hand, we have 
\[
m(a) \times \beta^{e(a)} < \beta^t \times \beta^{e(a)}
\]
following $m(a) < M$.
On the other hand, we have
\[
m(b) \times \beta^{e(b)} \geq \beta^{t-1} \times \beta^{e(b)} \geq \beta^t \times \beta^{e(a)}
\]
following the normal assumption.
%
Combining the two inequalities, we have $m(b) \times \beta^{e(b)} > m(a) \times \beta^{e(a)}$ and, thus, the two representations cannot possibly represent the same number.

\hfill {\bf End of proof for Lemma \ref*{lemma:NormalRepresentationIsUnique}.}
\end{lemma}

Not all numbers in Dekker's system admit normal representations. 
The most notable example is zero, which necessarily has all zero digits in the mantissa.
%
%
In fact, all numbers in the following set do not admit normal representations
\bgnEqn
\label{eqn:subnormalNumbers}
\left\{ x \,\vert\, x = m(x) \times \beta^{E_{\min}},\, \vert m(x)\vert < \beta^{t-1} \right\}.
\endEqn
%
%
\begin{definition}[Subnormal number]
A nonzero number is referred to as {\it subnormal} if it does not admit a normal representation in Dekker's system. 
\end{definition}

It is worth mentioning that in Dekker's system, subnormal numbers may still admit multiple 
representations. For example, with a three-digit mantissa system, both $010 \times \beta^{E_{\min}}$ and $001 \times \beta^{E_{\min}+1}$ represent the same subnormal number.

\begin{remark}[Subnormal numbers]
\label{rmk:subnormal_importance_and_practice}
%
Historically, with certain hardware and compilers, the default behavior regarding subnormal numbers is \textit{treat as zero} on input and \textit{flush to zero} on output, forfeiting the gradual underflow design from the IEEE standard. 
Operations involving subnormal numbers often incur heavy performance penalty as they are handled via microcode offloading. 
However, for reduced precision data formats, it is important to give subnormal numbers the proper treatment in hardware since they occupy a large portion of the number system and often cover an even larger portion of the representable range.
%
%
Interested readers may find more discussions on the merit of subnormal numbers in \cite[Appendix A]{overton2025numerical} and \cite{demmel1984underflow}.

\hfill {\bfseries{\itshape End of Remark} \ref*{rmk:subnormal_importance_and_practice}.}
\end{remark}

Following Lemma \ref{lemma:NormalRepresentationIsUnique}, one may impose uniqueness on the number system by keeping only the normal representations, i.e., by requiring $\vert m(x) \vert \geq \beta^{t-1}$.
This strategy, however, would leave an undesirable gap between zero and the smallest positive normal number $\beta^{t-1} \times \beta^{E_{\min}}$. 
%
%
This gap can be filled by including the subnormal numbers.
%
We refer to numbers in the resulting union shown in \eqref{eqn:protoIEEEfloatingPointSystem} as the {\it proto-IEEE} system. 
\bgnEqn
\label{eqn:protoIEEEfloatingPointSystem}
\left\{ x \,\vert\, x = m(x) \times \beta^{e(x)},\, \beta^{t-1} \leq \vert m(x) \vert < \beta^t,\, E_{\min} \leq e(x) \leq E_{\max} \right\} 
\union 
\left\{ x \,\vert\, x = m(x) \times \beta^{E_{\min}},\, \vert m(x) \vert < \beta^{t-1} \right\}
\endEqn

\begin{remark}[\texttt{Inf} and \texttt{NaN} in IEEE system]
\label{rmk:proto-IEEEnumberSystem}
The reason we refer to \eqref{eqn:protoIEEEfloatingPointSystem} as the proto-IEEE system, rather than simply the IEEE system, is that the IEEE system includes special members $\pm \texttt{Inf}$ and $\texttt{NaN}$'s as extensions to numbers. The former are used to handle overflow and the later are used to handle invalid arithmetic operations such as $\nicefrac{0}{0}$ and $\sqrt{-1}$.
In the IEEE system, representations with the largest exponent are dedicated to these extensions.

\hfill {\bfseries{\itshape End of Remark} \ref*{rmk:proto-IEEEnumberSystem}.}
\end{remark}

Our main interest in exploring Dekker's system is to understand the behavior of floating point numbers in numerical procedures, rather than their realization on hardware.
For our purpose, the analytical results obtained in Dekker's system are transferrable to the IEEE system. 
To give credence to this claim, 
%
we first establish the fact that Dekker's system and the proto-IEEE system contain the same numbers, which is provided in the following lemma using induction.
%
\begin{lemma}[Equivalence]
\label{lemma:equivalence_Dekker_proto-IEEE} 
The two number sets from Dekker's system \eqref{eqn:DekkerFloatingPointSystem} and the proto-IEEE system \eqref{eqn:protoIEEEfloatingPointSystem} are equivalent.
\\
\normalfont
{\bf Proof.}
Imagine a process of adding numbers to these two sets bin by bin.
For the first bin $E_{\min}$, Dekker's system has the following numbers 
\begin{equation}
\label{eqn:equivalence:Dekker:binEmin}
\left\{ x \,\vert\, x = m(x) \times \beta^{E_{\min}},\, \vert m(x)\vert < \beta^t \right\};
\end{equation}
the proto-IEEE system has the following numbers
\begin{equation}
\label{eqn:equivalence:proto-IEEE:binEmin}
\left\{ x \,\vert\, x = m(x) \times \beta^{E_{\min}},\, \beta^{t-1} \leq \vert m(x) \vert < \beta^t \right\} 
\union 
\left\{ x \,\vert\, x = m(x) \times \beta^{E_{\min}},\, \vert m(x) \vert < \beta^{t-1} \right\}.
\end{equation}
Clearly the two sets \eqref{eqn:equivalence:Dekker:binEmin} and \eqref{eqn:equivalence:proto-IEEE:binEmin} are the same.
Assuming the two systems contain the same numbers up to bin $E$.
For bin $E+1$, Dekker's system introduces the following numbers
\begin{equation}
\label{eqn:equivalence:binE+1}
\left\{ x \,\vert\, x = m(x) \times \beta^{E+1},\, \beta^{t-1} \leq \vert m(x)\vert < \beta^t \right\} 
\union 
\left\{ x \,\vert\, x = m(x) \times \beta^{E+1},\, \vert m(x)\vert < \beta^{t-1} \right\}.
\end{equation}
The first set of \eqref{eqn:equivalence:binE+1} is the same as the bin $E+1$ introduced by the proto-IEEE system;
the second set of \eqref{eqn:equivalence:binE+1} can be written as
$
\left\{ x \,\vert\, x = \big( m(x) \cdot \beta \big) \times \beta^{E},\, \vert m(x) \cdot \beta \vert < \beta^{t} \right\},
$
which has already been introduced by the previous bins. 
Therefore, the two systems contain the same numbers up to bin $E+1$.
By induction, the result stated in the lemma follows.

\hfill {\bf End of proof for Lemma \ref*{lemma:equivalence_Dekker_proto-IEEE}.}
\end{lemma}

To better appreciate the number systems, let us consider a concrete example where $\beta=2$, $t=3$, $E_{\min} = -3$, and $E_{\max} = 0$. 
Positive numbers that can be represented by this set of parameters are illustrated in Figure \ref{fig:number_system_example:entire_distribution}.
Notice that the distance between two consecutive numbers increases at the bin boundaries; specifically, it leaps by factor of $\beta$.
\begin{figure}[H]
\captionsetup{width=1\textwidth, font=small, labelfont=small}
\centering
\includegraphics[width=1\textwidth]{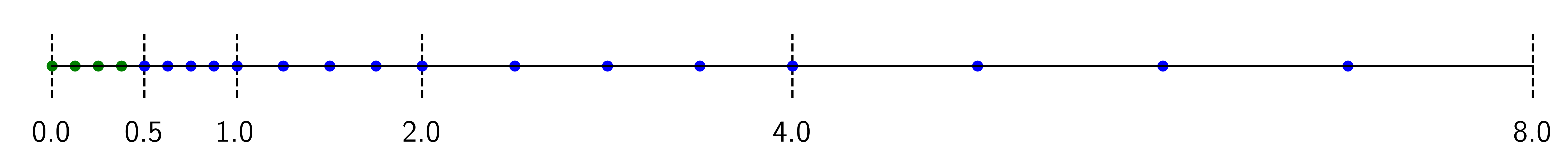}
\caption{Number distribution for a system with $\beta=2$, $t=3$, $E_{\min} = -3$, and $E_{\max} = 0$. The green and blue dots indicate subnormal and normal numbers, respectively.}
\label{fig:number_system_example:entire_distribution}
\end{figure}

Representations of these numbers in the proto-IEEE system grouped by bins are illustrated in Figure \ref{fig:number_system_example:proto_IEEE_bin_by_bin}. 
For the bin with exponent $E>E_{\min}$, the $\beta^{t-1}$ representations uniformly divide the interval $\big[\beta^{E+t-1}, \beta^{E+t}\big)$ with increment $\beta^E$.
%
Additionally, for the bin with exponent $E_{\min}$, the $\beta^{t}$ representations uniformly divide the interval $\big[0, \beta^{E_{\min} + t}\big)$ with increment $\beta^E_{\min}$. 
Half of the representations in bin $E_{\min}$ correspond to subnormal numbers.
\begin{figure}[H]
\captionsetup{width=1\textwidth, font=small, labelfont=small}
\centering
\includegraphics[width=1\textwidth]{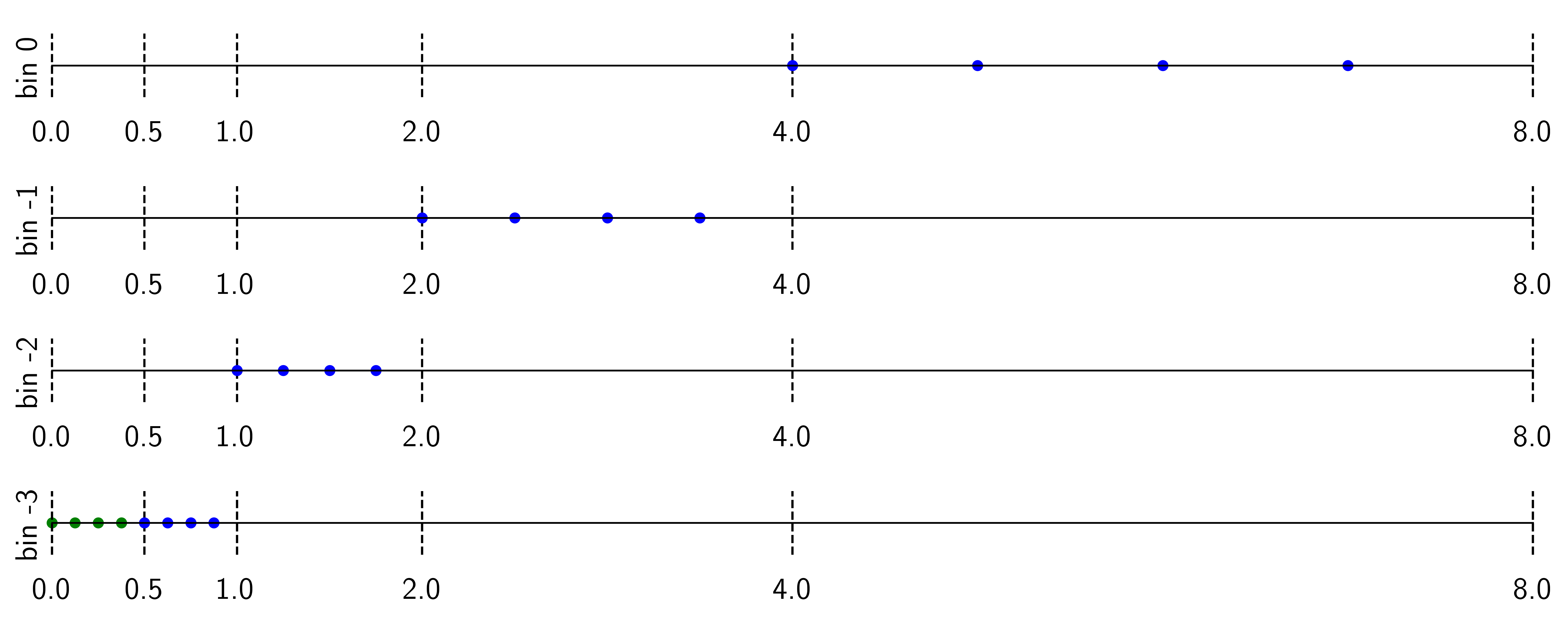}
\caption{Representations in the proto-IEEE system with $\beta=2$, $t=3$, $E_{\min} = -3$, and $E_{\max} = 0$, viewed bin by bin.}
\label{fig:number_system_example:proto_IEEE_bin_by_bin}
\end{figure}

In comparison, representations of these numbers in Dekker's system grouped by bins are illustrated in Figure \ref{fig:number_system_example:Dekker_bin_by_bin}.
For the bin with exponent $E$, the $\beta^t$ representations uniformly divide the interval $\big[0,\beta^{E+t}\big)$ with increment $\beta^E$.
Unlike the proto-IEEE system, bin $E_{\min}$ and the subnormal numbers does {\it not} need to be addressed specially in Dekker's system.
\begin{figure}[H]
\captionsetup{width=1\textwidth, font=small, labelfont=small}
\centering
\includegraphics[width=1\textwidth]{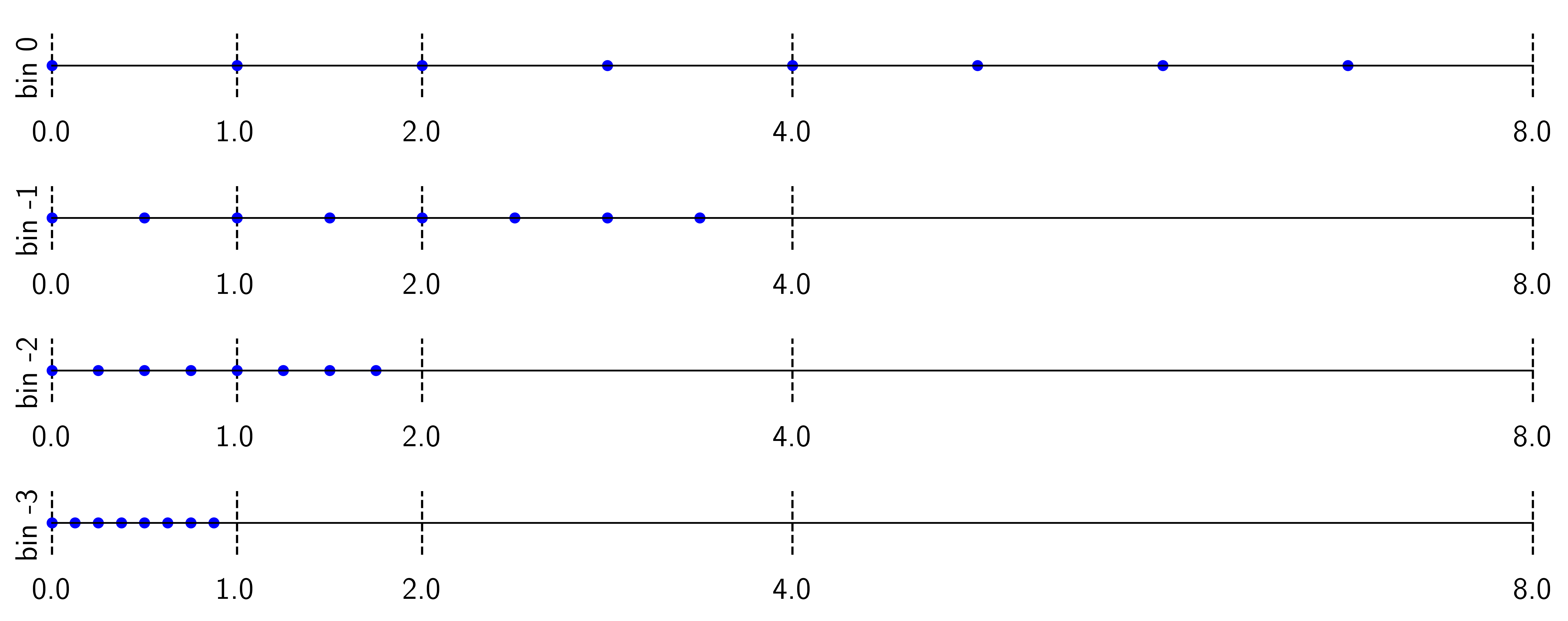}
\caption{Representations in Dekker's system with $\beta=2$, $t=3$, $E_{\min} = -3$, and $E_{\max} = 0$, viewed bin by bin.}
\label{fig:number_system_example:Dekker_bin_by_bin}
\end{figure}

In the context of this work, statements can often be made more general in Dekker's system than in the proto-IEEE system and offer more insights. Theorem \ref{thm:3opEFTadditionDekker} is one example, whose generality is exploited in the proof of Theorem \ref{thm:6opEFTadditionDekker}. 
Dekker's system can also offer simplicity because statements do not need to be specialized to subnormal numbers with Lemma \ref{lemma:relationBetweenFloatingPointRoundingAndMantissaIntegerRounding} being an example.
Lemma \ref{lemma:FergusonCounterpart} offers another example that illustrates the convenience and effectiveness of Dekker's system.

\begin{remark}[Nonuniqueness]
\label{rmk:nonuniqueness}
Admitting nonunique representations in a number system may appear uncomfortable at first sight, but in fact, even the IEEE system contains nonunique representations in the special case of \textit{zero}, which has two representations with different sign bits. 
In Dekker's system, there are two such zeros in each bin.
The IEEE standard \cite[Section 6.3]{IEEE2019standard} defines special rules to determine the sign of zero when ambiguity may arise. However, these are rather technical aspects that do not concern this work.

\hfill {\bfseries{\itshape End of Remark} \ref*{rmk:nonuniqueness}.}
\end{remark}

\subsection{Arithmetic and rounding}
In addition to the numbers and representations, a floating point number system $\mathbb F$ is usually equipped with a set of arithmetic operations.
Definitions of these arithmetic operations can be derived from their counterparts in the real number system $\mathbb R$.
However, $\mathbb F$ is usually not closed with respect to the arithmetic operations defined in $\mathbb R$.
For this reason, some mechanism is needed to close the floating point number system with respect to 
arithmetic operations. 
For 
floating point number systems, this can be {\it mostly} achieved by the operation of rounding with the possible exception of overflow.

Rounding maps $x \in \mathbb R$ to $fl(x) \in \mathbb F$ where $fl(x)$ is close to $x$. 
In Dekker's system, when $x$ overflows, i.e., falls outside the range of $\mathbb F$, rounding maps it to the nearest element in $\mathbb F$, i.e., one of $\pm (M-1) \times \beta^{E_{\max}}$ depending on the sign of $x$. 
This is sufficient to close the system for valid arithmetic operations defined in $\mathbb R$.
One may refer to the handling of overflow in Dekker's system as {\it clipping} or {\it clamping}, which as a concept appears often in computer engineering practice. 
In contrast, the IEEE system handles overflow with $\pm \texttt{Inf}$.

The floating point arithmetic operations defined over a floating point number system $\mathbb F$ are simply their counterparts in the real system $\mathbb R$ followed by the rounding operation.
Two common notions pertaining to rounding are given below. 
\begin{definition}[Faithful]
For $x \in \mathbb R$, rounding is faithful if $fl(x) \in \mathbb F$ is either the largest element of $\mathbb F$ no larger than $x$ or the smallest element of $\mathbb F$ no smaller than $x$. As a consequence, if $x \in \mathbb F$, then $fl(x) = x$ under faithfulness.
\end{definition}
\begin{definition}[Optimal]
For $x \in \mathbb R$, rounding is optimal if $fl(x) \in \mathbb F$ is the nearest element to $x$. 
Optimal rounding may also be referred to as nearest rounding.
\end{definition}
\noindent

When $x$ lies at the exact middle of two consecutive members of $\mathbb F$, a tie-breaking criterion is needed. 
%
%
Possible tie-breaking criterions include rounding to even (or odd) last mantissa digit and rounding away from (or towards) zero.

\begin{remark}[Tie-breaking]
\label{rmk:tie-breaking-ambiguity}

Nearest rounding with the tie-breaking criterion of rounding to even last mantissa digit is the default rounding mode for binary formats in the IEEE standard \cite[Section 4.3]{IEEE2019standard} and is prevalently supported on hardware.
It may not always be possible to round to a number with even last mantissa digit such as in the extreme case when $t=1$. A fallback criterion such as rounding away from zero may be used in such situations as suggested by the IEEE standard.

In Dekker's system, rounding to even or odd last mantissa digit can be ambiguous because the same number may admit representations with both even and odd last mantissa digit, e.g., both $100 \times \beta^{-2}$ and $001 \times \beta^{0}$ 
from the system illustrated in Figure \ref{fig:number_system_example:Dekker_bin_by_bin} 
represent the same number.
To remove this ambiguity, given $x \in \mathbb R$, tie-breaking should be performed on a representation of $x$ that has $t$ whole digits either with a non-zero leading digit or with exponent $E_{\min}$, possibly both. 

Using the illustration in Figure \ref{fig:number_system_example:Dekker_bin_by_bin}, this simply means when $x$ can be enclosed in multiple intervals $\left[000, 111\right) \times \beta^E$, where $E_{\min} \leq E \leq E_{\max}$, tie-breaking, and rounding in general, should be performed on a representation of $x$ with the smallest exponent $E$ possible.
This procedure will also deliver the rounding outcome consistent with the proto-IEEE system under nearest rounding with the same tie-breaking criterion.

In the following, when referring to tie-breaking with rounding to even or odd in Dekker's system, the above procedure is implied and, more generally, the consistency of rounding outcome between Dekker's system and the proto-IEEE system is assumed.
For tie-breaking criterions that bias towards a direction, there is no ambiguity in Dekker's system and the rounding outcome is consistent between the two systems.

\hfill {\bfseries{\itshape End of Remark} \ref*{rmk:tie-breaking-ambiguity}.}
\end{remark}

All commonly encountered rounding modes are at least faithful. 
%
%
In the following, the rounding mode, including the tie-breaking criterion if needed, is assumed to be fixed throughout the procedure under discussion and deterministic 
so that the rounding outcome is unique and depends on the input only.
To facilitate the discussion below, we assume that the following rule holds for floating point arithmetic operations.
\begin{thmRule}[The {\it as if} rule] 
\label{Rule:AsIf} 
The floating point arithmetic operation under consideration are performed {\it as if} an intermediate result correct to infinite precision and with unbounded range has been produced first, and then rounded to fit the destination format.
\end{thmRule}

%
\noindent Rule \ref{Rule:AsIf} simplifies the discussion because it decouples rounding from the arithmetic operations and, thus, enables statements regarding rounding behavior to be made without specializing to a particular arithmetic operation.
%
%
This rule is stipulated in the IEEE-754 standard \cite{IEEE2019standard} for addition, subtraction, multiplication, division, square root, and fused multiply add (FMA), and is widely supported on modern hardware.

It is easy to see that under Rule \ref{Rule:AsIf} and with consistent rounding modes, floating point arithmetic operations in Dekker's system and the proto-IEEE system deliver the same outcome since, first, the calculation of the infinitely precise intermediate result can be understood as being carried out in $\mathbb R$, i.e., independent of the floating point representation, and, second, the subsequent rounding outcome depends only on the rounding mode and number distribution in the system, which is the same according to Lemma \ref{lemma:equivalence_Dekker_proto-IEEE}.

\subsection{Preparatory results on rounding}

Below, we provide a few lemmas on rounding that will be needed in the subsequent discussion.

\begin{lemma}[Symmetry] 
\label{lemma:Symmetry} 
Assume $\mathbb F$ is symmetric, i.e., 
the negative of any element of $\mathbb F$ is also in $\mathbb F$.
Then $\forall x \in \mathbb R$, we have $fl(-x) = -fl(x)$ under nearest rounding with a tie-breaking criterion that is consistent between the rounding of $-x$ and $x$. 
Such tie-breaking criteria include rounding to even or odd last mantissa digit or rounding towards or away from zero.
\\
\normalfont
{\bf Proof.}
For the special case $x=0$, we have $fl(x) = 0$ since $0 \in \mathbb F$ and the relation $fl(-x) = -fl(x)$ clearly holds. 
When $x \neq 0$, we need only to consider the case $x>0$ because both $\mathbb R$ and $\mathbb F$ are symmetric.
In this case, we have 
\[
\vert fl(x) - x \vert \leq \vert a - x \vert,\, \forall a \in \mathbb F
\]
under nearest rounding, and hence
\[
\left\vert \big(-fl(x)\big) - (-x) \right\vert \leq \vert (-a) - (-x) \vert,\, \forall (-a) \in \mathbb F.
\]
In other words, $-fl(x)$ is either \textit{the} closest element in $\mathbb F$ to $-x$, in which case we have $fl(-x) = -fl(x)$ by the definition of nearest rounding, or one of the two closest elements when $-x$ lies at the exact middle of two adjacent elements of $\mathbb F$.
In the later scenario, $fl(-x) = -fl(x)$ follows the consistency in the tie-breaking criterion.

\hfill {\bf End of proof for Lemma \ref*{lemma:Symmetry}.}
\end{lemma}

We note here that Lemma \ref{lemma:Symmetry} does not hold under tie-breaking criteria such as rounding towards \texttt{-Inf} or \texttt{+Inf}.

\begin{lemma}[Monotonicity]
\label{lemma:Monotonicity}
For $x, y \in \mathbb R$, if $\vert x \vert \leq \vert y \vert$, then $\vert fl(x) \vert \leq \vert fl(y) \vert$ under nearest rounding. \\
\normalfont
{\bf Proof.}
We need only to consider the case $0 \leq x \leq y$ since, according to Lemma \ref{lemma:Symmetry}, rounding is symmetric under nearest rounding. 
%
Moreover, if $x=y$, then $fl(x)=fl(y)$ since rounding leads to unique outcome. The result stated in this lemma clearly holds in this case.
Therefore, we can further restrict ourselves to considering only $0 \leq x < y$.

By the definition of nearest rounding, we have 
\begin{subequations}
\label{lemma:Monotonicity:nearest}
\begin{align}
\vert fl(x) - x \vert & \leq \vert a - x \vert, \, \forall a \in \mathbb F; 
\label{lemma:Monotonicity:nearest->a} \\
\vert fl(y) - y \vert & \leq \vert b - y \vert, \, \forall b \in \mathbb F,
\label{lemma:Monotonicity:nearest->b}
\end{align}
\end{subequations}
and, consequently, 
\bgnEqn
\label{lemma:Monotonicity:nearestAddTogether}
\vert fl(x) - x \vert + \vert fl(y) - y \vert 
\leq 
\vert a - x \vert + \vert b - y \vert, \, \forall a,b \in \mathbb F.
\endEqn
We proceed by contradiction in the separate cases below. 

\begin{adjustwidth}{1.5em}{}
\setlength{\parindent}{0em}
\textcolor{kwdcolor}{Case 1:} if $fl(x) > fl(y) \geq x$, this \textcolor{kwdcolor}{contradicts with \eqref{lemma:Monotonicity:nearest->a}} as $fl(y)$ is strictly closer to $x$ than $fl(x)$.

\textcolor{kwdcolor}{Case 2(a):} if $fl(x) > fl(y)$, $fl(y) < x$, and $fl(x) \leq y$, this \textcolor{kwdcolor}{contradicts with \eqref{lemma:Monotonicity:nearest->b}} as $fl(x)$ is strictly closer to $y$ than $fl(y)$.

\textcolor{kwdcolor}{Case 2(b):} if $fl(x) > fl(y)$, $fl(y) < x$, and $fl(x) > y$, we have 
\begin{align*}
\big(y - fl(y)\big) + \big(fl(x) - x\big) 
& = \big(y - x + x - fl(y)\big) + \big(fl(x) - y + y - x\big) \\
& = 2\cdot(y-x) + \big(x - fl(y)\big) + \big(fl(x) - y\big) \\
& > \big(x - fl(y)\big) + \big(fl(x) - y\big)
\end{align*}
following $x<y$, which \textcolor{kwdcolor}{contradicts with \eqref{lemma:Monotonicity:nearestAddTogether}}.
\end{adjustwidth}
The above cases cover all possible positions of $fl(x)$ and $fl(y)$ in relation to $x$ and $y$, which all leads to contradiction if $fl(x) > fl(y)$. Therefore, the result stated in this lemma must hold.

\hfill {\bf End of proof for Lemma \ref*{lemma:Monotonicity}.}
\end{lemma}

\begin{lemma}
\label{lemma:floor&ceiling_are_the_closest_for_normal_numbers} 
If $\tilde z > 0 \in \mathbb R$ can be expressed as 
$\tilde z = \zeta \times \beta^{e(z)}$ with 
$\beta^{t-1} \leq \zeta \leq \beta^t$ and 
$E_{\min} \leq e(z) < E_{\max}$,
then the number in $\mathbb F$ closest to $\tilde z$ is either $\lfloor \zeta \rfloor \times \beta^{e(z)}$ or $\lceil \zeta \rceil \times \beta^{e(z)}$, 
where, for $\zeta > 0 \in \mathbb R$, $\lfloor \zeta \rfloor$ and $\lceil \zeta \rceil$ are defined as 
$\lfloor \zeta \rfloor = \max \{ i \, \vert \, i \in \mathbb Z, i \leq \zeta \}$ 
and
$\lceil  \zeta \rceil  = \min \{ i \, \vert \, i \in \mathbb Z, i \geq \zeta \}$, respectively.
\\
\normalfont
{\bf Proof.}
The conditions $\beta^{t-1} \leq \zeta \leq \beta^t$ and $E_{\min} \leq e(z) < E_{\max}$ ensure that both $\lfloor \zeta \rfloor \times \beta^{e(z)}$ and $\lceil \zeta \rceil \times \beta^{e(z)}$ are in $\mathbb F$.
Consider now an arbitrary $w>0$ in $\mathbb F$ with a representation $w = m(w) \times \beta^{e(w)}$. 

If $e(w) < e(z)$, we have $m(w) \times \beta^{e(w)} \leq (\beta^t - 1) \times \beta^{e(z)-1} < \beta^{t-1} \times \beta^{e(z)} \leq \lfloor \zeta \rfloor \times \beta^{e(z)} \leq \zeta \times \beta^{e(z)}$. Therefore, $\lfloor \zeta \rfloor \times \beta^{e(z)}$ is closer to $\tilde z$ than $w$ in this case.

If $e(w) \geq e(z)$, we can write $w = m(w) \cdot \beta^h \times \beta^{e(z)}$, 
where $h = e(w) - e(z) \geq 0$ and, therefore, $m(w) \cdot \beta^h \in \mathbb Z$.
Since $\lfloor \zeta \rfloor$ and $\lceil \zeta \rceil$ are consecutive integers, we have either 
$m(w) \cdot \beta^h \leq \lfloor \zeta \rfloor \leq \zeta$ or 
$m(w) \cdot \beta^h \geq \lceil \zeta \rceil \geq \zeta$.
Therefore, either $\lfloor \zeta \rfloor \times \beta^{e(z)}$ or $\lceil \zeta \rceil \times \beta^{e(z)}$ is closest to $\tilde{z}$ in this case.

\hfill {\bf End of proof for Lemma \ref{lemma:floor&ceiling_are_the_closest_for_normal_numbers}.}
\end{lemma}

We note here that in the statement of Lemma \ref{lemma:floor&ceiling_are_the_closest_for_normal_numbers}, $\lfloor \zeta \rfloor \times \beta^{e(z)}$ and $\lceil \zeta \rceil \times \beta^{e(z)}$ are referred to as {\it numbers} rather than {\it representations}.
Specifically, $\lceil \zeta \rceil \times \beta^{e(z)}$ is not necessarily a valid representation in $\mathbb F$, which happens when $\zeta > \beta^t - \tfrac{1}{2}$.
Moreover, Lemma \ref{lemma:floor&ceiling_are_the_closest_for_normal_numbers} can be easily extended to include real numbers $\left\{\tilde{z} \,\vert\, \tilde{z} = \zeta \times \beta^{E_{\max}},\, \beta^{t-1} \leq \zeta \leq \beta^t - 1 \right\}$ in the statement.
We opted for the weaker version for the clarity of its statement, which is sufficient for the subsequent proofs. 

%
%
Lemma \ref{lemma:floor&ceiling_are_the_closest_for_normal_numbers} covers the case where $\tilde z > 0$. 
The following corollary extends the result to $\tilde z < 0$.
It can be proved by applying Lemma \ref{lemma:floor&ceiling_are_the_closest_for_normal_numbers} to $- \tilde z$ and using the argument that $\mathbb R$ is symmetric with respect to zero.

\begin{corollary}
If $\tilde z < 0 \in \mathbb R$ can be expressed as 
$\tilde z = \zeta \times \beta^{e(z)}$ with 
$ - \beta^{t-1} \geq \zeta \geq - \beta^t$ and 
$E_{\min} \leq e(z) < E_{\max}$,
then the number in $\mathbb F$ closest to $\tilde z$ is either $- \lfloor - \zeta \rfloor \times \beta^{e(z)}$ or $- \lceil - \zeta \rceil \times \beta^{e(z)}$.
\end{corollary}

For convenience, we introduce the notation $\mathbb D$ to refer to Dekker's system for the subsequent discussion to distinguish it from a generic floating point number system denoted by $\mathbb F$.
The following lemma allows us to reduce the reasoning of floating point rounding to integer rounding, which is considerably simpler. 
It is relied upon heavily in the subsequent proofs.
\begin{lemma}
\label{lemma:relationBetweenFloatingPointRoundingAndMantissaIntegerRounding} 
Under nearest rounding, given $\tilde{z} \in \mathbb R$ and 
$\vert \tilde{z} \vert \leq (M-1) \times \beta^{E_{\max}}$, 
i.e., $\tilde{z}$ does not overflow in Dekker's system $\mathbb D$, if $z = fl(\tilde{z})$ admits a representation $z = m(z) \times \beta^{e(z)}$, then expressing $\tilde{z}$ as $\tilde{z} = \zeta \times \beta^{e(z)}$, we have 
\begin{equation}
\label{eqn:m(z)=round(zeta)}
m(z) = round(\zeta),
\end{equation}
where $round(\zeta) \in \mathbb Z$ gives the nearest integer to $\zeta$ with a tie-breaking criterion, if needed, consistent with that used for $fl(\tilde{z})$.
\\
\normalfont
{\bf Proof.}
It suffices to consider only the case $\tilde{z} > 0$ because both $\mathbb R$ and $\mathbb D$ are symmetric and the case $\tilde{z} = 0$ is trivial.

We start by showing that $round(\zeta) \times \beta^{e(z)} \in \mathbb D$. 
Since $z = m(z) \times \beta^{e(z)}$ is a valid representation in $\mathbb D$, we have $e(z) \in \left[ E_{\min} , E_{\max} \right]$. 
Thus, we only need to show that $round(\zeta) \leq M-1$.
First, we consider the special case where $e(z) = E_{\max}$: the condition $\vert \tilde{z} \vert \leq (M-1) \times \beta^{E_{\max}}$ implies that $\zeta \leq M-1$ and, therefore, we have $round(\zeta) \leq M-1$ as desired.
When $e(z) < E_{\max}$, we separate the discussion into three cases depending on the size of $\zeta$.

\begin{adjustwidth}{1.5em}{}
\setlength{\parindent}{0em}
\textcolor{kwdcolor}{Case 1: if $\zeta < M - \nicefrac{1}{2}$}, then we have the desired result $round(\zeta) \leq M-1$.

\textcolor{kwdcolor}{Case 2: if $\zeta > M - \nicefrac{1}{2}$}, then we have 
$
\zeta - m(z) \geq \zeta - (M-1) > \vert \zeta - M \vert.
$
In other words, $M \times \beta^{e(z)} = \tfrac{M}{\beta} \times \beta^{e(z)+1} \in \mathbb D$ 
is closer to $\tilde{z}$ than any valid representation in $\mathbb D$ with exponent $e(z)$.
%
This contradicts the condition that $z$ admits a representation in $\mathbb D$ with exponent $e(z)$ since $M \times \beta^{e(z)}$ is not a valid representation itself. 
%
%
Therefore, this is a null case.

\textcolor{kwdcolor}{Case 3: if $\zeta = M - \nicefrac{1}{2}$}, then $M-1$ and $M$ are the two closest integers to $\zeta$ with equal distance. 
Moreover, from Lemma \ref{lemma:floor&ceiling_are_the_closest_for_normal_numbers}, 
we know that $(M-1) \times \beta^{e(z)}$ and $\tfrac{M}{\beta} \times \beta^{e(z)+1}$ 
are the two floating point numbers in $\mathbb D$ closest to $\tilde{z}$ with equal distance. 
Furthermore, the condition that $z$ admits a representation with exponent $e(z)$ implies that the tie-breaking criterion rounds $\tilde{z}$ to $(M-1) \times \beta^{e(z)}$ instead of $\tfrac{M}{\beta} \times \beta^{e(z)+1}$. 
Since the tie-breaking criterion is assumed to be consistent when rounding $\tilde{z}$ and $\zeta$, 
we have $round(\zeta) = M-1$ as desired.
%
\end{adjustwidth}
\noindent We now have established that $round(\zeta) \times \beta^{e(z)} \in \mathbb D$.
Next, we proceed to show that $\tilde{z}$ indeed rounds to $round(\zeta) \times \beta^{e(z)}$ by contradiction.
\textcolor{kwdcolor}{Assuming} $m(z) \neq round(\zeta)$, we have $\vert m(z) - round(\zeta) \vert \geq 1$ since they are different integers. 
This leads to
\begin{equation}
\label{eqn:comparing_distance_to_zeta_from_m(z)_and_round(zeta)}
\vert m(z) - \zeta \vert 
\geq \vert m(z) - round(\zeta) \vert - \vert \zeta - round(\zeta) \vert
\geq 1 - \tfrac{1}{2}
= \tfrac{1}{2},
\end{equation}
where we have used $\vert \zeta - round(\zeta) \vert \leq \tfrac{1}{2}$.

If either 
$\vert m(z) - \zeta \vert > \tfrac{1}{2}$ or 
$\vert round(\zeta) - \zeta \vert < \tfrac{1}{2}$ holds, we have 
$\vert m(z) - \zeta \vert > \vert round(\zeta) - \zeta \vert$.
In other words, $round(\zeta) \times \beta^{e(z)}$ would be closer to $\tilde{z}$ than $m(z) \times \beta^{e(z)}$, which contradicts the condition that $m(z) \times \beta^{e(z)}$ is a representation of $fl(\tilde{z})$.
The only remaining case is $\vert m(z) - \zeta \vert = \tfrac{1}{2} =\vert round(\zeta) - \zeta \vert$, which would contradict the condition that the tie-breaking criterion is consistent when rounding $\tilde{z}$ and $\zeta$ since $\tilde{z}$ rounds to $m(z) \times \beta^{e(z)}$ instead of $round(\zeta) \times \beta^{e(z)}$.
Hence, the \textcolor{kwdcolor}{assumption} $m(z) \neq round(\zeta)$ must be incorrect and \eqref{eqn:m(z)=round(zeta)} must hold. 

\hfill {\bf End of proof for Lemma \ref*{lemma:relationBetweenFloatingPointRoundingAndMantissaIntegerRounding}.}
\end{lemma}

\subsection{Preparatory results on addition}

Below, we give a few results on the addition operation that are useful for the subsequent discussion.
The following lemma, although intuitive, is still worth a rigorous statement. 
\begin{lemma}
\label{lemma:e(z)BoundFromAbove}
Given $x, y \in \mathbb D$ with $e(x) \geq e(y)$ and $z$ obtained via $z = fl(x+y)$, then $z$ admits a representation in $\mathbb D$ with $e(z) \leq e(x) + 1$.\\
\normalfont
{\bf Proof.}
Since $x, y \in \mathbb D$ and $e(x) \geq e(y)$, we have 
\begin{equation*}
\vert x + y \vert
\leq \vert x \vert + \vert y \vert
\leq 2 \cdot \left( M - 1 \right) \times \beta^{e(x)} 
\leq \left( M - 1 \right) \times \beta^{e(x)+1}.
\end{equation*}
Using Lemma \ref{lemma:Monotonicity} (Monotonicity) and recognizing that $\left( M - 1 \right) \times \beta^{e(x)+1} \in \mathbb D$ if $e(x) < E_{\max}$, we arrive at
\begin{equation*}
\vert z \vert = \vert fl(x + y) \vert 
\leq \left( M - 1 \right) \times \beta^{e(x)+1}
\end{equation*}
in this case and, hence, $z$ admits a representation in $\mathbb D$ with $e(z) \leq e(x) + 1$.

If $e(x) = E_{\max}$, it is necessary that $e(z) \leq e(x)$ and therefore the lemma still holds.
Overflow is not a special concern here because in Dekker's system overflow leads to clipping. If overflow happens, it is necessary that $e(x) = E_{\max}$ and $z$ would take the value $\left( M - 1 \right) \times \beta^{E_{\max}}$.

\hfill {\bf End of proof for Lemma \ref*{lemma:e(z)BoundFromAbove}.}
\end{lemma}

The simple result below is interesting on its own and handy in the subsequent proofs as well.
\begin{lemma}
\label{lemma:z=x+yLeadsToExactSubtraction}
For $x, y \in \mathbb D$ and $z = fl(x+y)$, if $z=x+y$, i.e., there is no rounding error committed when forming $z$, then the subsequent operation $z-x$ or $z-y$ commits no rounding error either. \\
\normalfont
{\bf Proof.}
Since $z=x+y$, we have $z-x=y \in \mathbb D$, and, thus, $fl(z-x)=fl(y)=y$ commits no rounding error. Similar argument can be made for $z-y$.

\hfill {\bf End of proof for Lemma \ref*{lemma:z=x+yLeadsToExactSubtraction}.}
\end{lemma}

The following result is the counterpart of \cite[Theorem~2.4]{higham2002accuracy} stated in Dekker's system, which is more general in its statement and simpler in its proof thanks to the adoption of Dekker's system. 
Interested readers may also compare to its statement and proof from \cite{ferguson1995exact} and find its application in evaluating exponential functions there.
Faithful rounding is sufficient for the proof given below. 
\begin{lemma}
\label{lemma:FergusonCounterpart}
For $x, y \in \mathbb D$ and $z = fl(x+y)$, assuming overflow does not happen, if $x$, $y$, and $z$ admit representations in $\mathbb D$ such that $e(z) \leq min(e(x),e(y))$, then $z = x + y$, i.e., the addition is exact. \\
\normalfont
{\bf Proof.}
We can express $x+y$ as $\zeta \times \beta^{e(z)}$, where 
\begin{equation}
\label{eqn:zeta_for_(x+y)_in_e(z)}
\zeta = m(x) \cdot \beta^{e(x)-e(z)} + m(y) \cdot \beta^{e(y)-e(z)}.
\end{equation}
Since $e(x) \geq e(z)$ and $e(y) \geq e(z)$, we have $\zeta \in \mathbb Z$. 

If $\vert \zeta \vert \geq M$, $\zeta \times \beta^{e(z)}$ {\it cannot} possibly round to a representation with exponent $e(z)$. 
To see this more clearly and without loss of generality, assuming $\zeta \geq M$, then $\zeta \times \beta^{e(z)}$ would be closer to $\frac{M}{\beta} \times \beta^{e(z)+1}$ than to any valid representation in $\mathbb D$ with exponent $e(z)$. 
Thus, we have $\vert \zeta \vert < M$ and hence $x+y = \zeta \times \beta^{e(z)} \in \mathbb D$. The statement in the lemma holds following faithful rounding.
%

\hfill {\bf End of proof for Lemma \ref*{lemma:FergusonCounterpart}.}
\end{lemma}

The following result is a direct application of Lemma \ref{lemma:FergusonCounterpart}. It is useful for including subnormal numbers in the discussion such as in the case of Corollary \ref{col:bound_zz_by_exponent_e(z)}.
\begin{lemma}
\label{lemma:additionIsExactWhenEmin}
For $x, y \in \mathbb D$ and $z = fl(x+y)$, if $z$ admits a representation with exponent $E_{\min}$, then $z = x + y$, i.e., the addition is exact. \\
\normalfont
{\bf Proof.}
For the representation of $z$ with exponent $E_{\min}$, it is necessary that $e(x) \geq e(z)$ and $e(y) \geq e(z)$, hence, Lemma \ref{lemma:FergusonCounterpart} applies.

\hfill {\bf End of proof for Lemma \ref*{lemma:additionIsExactWhenEmin}.}
\end{lemma}
As a {\it special} case of the above lemma, we have that when $z = fl(x+y)$ is a {\it subnormal} number, 
the addition $x+y$ is exact, which is often the version stated in literature (e.g., \cite[Theorem 3.4.1]{hauser1996handling}). 

\section{Error free transformation over addition}
\label{sec:EFTunderAddition}

The following theorem is one of the building blocks of this work. 
It has already appeared in Dekker's work \cite{dekker1971floating}, but lacks detail there.
We provide a detailed version of the proof below and build on it for our later results.
Interested readers may find in \cite{kahan1965pracniques, moller1965quasi, knuth1997art_vol2} other versions of this result 
in a less general setting.

\begin{theorem}[3op EFT over addition in Dekker's system]
\label{thm:3opEFTadditionDekker}
Under Rule \ref{Rule:AsIf} and nearest rounding, $\forall x, y \in \mathbb D$ with $e(x) \geq e(y)$, the following formulas achieve EFT $z + zz = x + y:$
\begin{subequations}
\label{DekkerAdditionEFT3op}
\begin{align}
 z & = fl(x + y); \label{DekkerAdditionEFT3op->z} \\
 w & = fl(z - x); \label{DekkerAdditionEFT3op->w} \\
zz & = fl(y - w). \label{DekkerAdditionEFT3op->zz}
\end{align}
\end{subequations}
\normalfont
\noindent{\bf Proof.}
The gist of the proof below is to show that under the given conditions, neither \eqref{DekkerAdditionEFT3op->w} nor \eqref{DekkerAdditionEFT3op->zz} commits rounding error. 
%
%
Defining $d, f, g \in \mathbb Z$ as follows,
\begin{subequations}
\label{eqn:Definition_d_f_g}
\begin{align}
d & = e(x) - e(y); \label{eqn:Definition_d} \\
f & = e(x) - e(z); \label{eqn:Definition_f} \\
g & = e(y) - e(z), \label{eqn:Definition_g}
\end{align}
\end{subequations}
we have $d \geq 0$ from the condition $e(x) \geq e(y)$ of the theorem
and the following equivalent ways of expressing $x$, $y$, and $z$:
\begin{subequations}
\label{eqn:expressions_x_y_z}
\begin{alignat}{4}
x 
& \enskip=\enskip & m(x)                  \times \beta^{e(x)} 
& \enskip=\enskip & m(x) \cdot \beta^{ d} \times \beta^{e(y)}
& \enskip=\enskip & m(x) \cdot \beta^{ f} \times \beta^{e(z)} &; 
\label{eqn:expressions_x} \\
y 
& \enskip=\enskip & m(y) \cdot \beta^{-d} \times \beta^{e(x)} 
& \enskip=\enskip & m(y)                  \times \beta^{e(y)}
& \enskip=\enskip & m(y) \cdot \beta^{ g} \times \beta^{e(z)} &; 
\label{eqn:expressions_y} \\
z 
& \enskip=\enskip & m(z) \cdot \beta^{-f} \times \beta^{e(x)} 
& \enskip=\enskip & m(z) \cdot \beta^{-g} \times \beta^{e(y)} 
& \enskip=\enskip & m(z)                  \times \beta^{e(z)} &.
\label{eqn:expressions_z}
\end{alignat}
\end{subequations}
%
%
We note here that the above equalities are stated in $\mathbb R$ rather than in $\mathbb D$ because the mantissas may be out of range or no longer be integers. 
From \eqref{eqn:expressions_x} and \eqref{eqn:expressions_y}, we have
\begin{equation}
\label{eqn:expression_x+y}
x + y = \left( m(x) \cdot \beta^f + m(y) \cdot \beta^g \right) \times \beta^{e(z)}.
\end{equation}
Using Lemma \ref{lemma:relationBetweenFloatingPointRoundingAndMantissaIntegerRounding}, we have 
\begin{equation}
\label{eqn:expression_m(z)}
m(z) = round \left( m(x) \cdot \beta^f + m(y) \cdot \beta^g \right).
\end{equation}
We introduce 
\begin{equation}
\label{zetaDefinition}
\zeta = m(x) \cdot \beta^f + m(y) \cdot \beta^g
\end{equation}
to simplify the notation below.

We proceed by showing that \eqref{DekkerAdditionEFT3op->w}, i.e., the calculation of $w$, commits no rounding error in all cases categorized by the signs of $f$ and $g$. 
From Lemma \ref{lemma:e(z)BoundFromAbove}, we know that $z$ admits a representation with $e(z) \leq e(x) + 1$ in Dekker's system. 
We choose to work with such a representation of $z$ for the discussion below, which is needed (only) in Case 4.
Moreover, we assume that overflow does not happen when forming $z$ in \eqref{DekkerAdditionEFT3op->z} for the cases below. 
The special scenario of overflow is addressed immediately after the discussion of the cases below.

\begin{adjustwidth}{1.5em}{}
\setlength{\parindent}{0em}
\textcolor{kwdcolor}{Case 1: if $f \geq 0$ and $g \geq 0$}, we have $\zeta \in \mathbb Z$ and, therefore, $m(z) = \zeta$. 
In other words, $z = x+y$ is formed without rounding error in \eqref{DekkerAdditionEFT3op->z}.
According to Lemma \ref{lemma:z=x+yLeadsToExactSubtraction}, \eqref{DekkerAdditionEFT3op->w} also commits no rounding error.

\textcolor{kwdcolor}{Case 2: if $f \geq 0$ and $g < 0$}, $\zeta$ is no longer necessarily an integer.
From \eqref{eqn:expressions_x_y_z}, we can write
\begin{equation}
\label{eqn:(z-x)inCase:f>=0;g<0}
z - x = \left( m(z) - m(x) \cdot \beta^f \right) \times \beta^{e(z)}.
\end{equation}
Denoting $\omega = m(z) - m(x) \cdot \beta^f$, we have $\omega \in \mathbb Z$ since $f \geq 0$ 
and
\begin{align*}
\left\vert \omega \right\vert 
& = \left\vert m(z) - m(x) \cdot \beta^f - m(y) \cdot \beta^g + m(y) \cdot \beta^g \right\vert \\
& \leq \left\vert round(\zeta) - \zeta \right\vert + \left\vert m(y) \cdot \beta^g \right\vert \\
& \leq \tfrac{1}{2} + \tfrac{M-1}{\beta} 
\leq \tfrac{M}{2}.
\end{align*}
We have $\tfrac{M}{2} \leq M-1$ if $\beta \geq 2$ and $t \geq 1$, which are satisfied for any natural floating point number system. 
It follows that $z - x = \omega \times \beta^{e(z)} \in \mathbb D$, and hence $w = fl(z-x)$ commits no rounding error.

\textcolor{kwdcolor}{Case 3: if $f < 0$ and $g \geq 0$}, we have $d = f - g < 0$, which contradicts with the condition $d \geq 0$. 

\textcolor{kwdcolor}{Case 4: if $f < 0$ and $g < 0$}, recall that $e(z) \leq e(x) + 1$, we have $f \geq -1$ and, therefore, $f=-1$.
In general, we can write 
\begin{equation}
\label{eqn:(z-x)inCase:f<0;g<0}
z - x = \left( m(z) \cdot \beta^{-f} - m(x) \right) \times \beta^{e(x)}.
\end{equation}
Denoting $\omega = m(z) \cdot \beta^{-f} - m(x)$, we have $\omega \in \mathbb Z$ and
\begin{subequations}
\label{eqn:muBoundAboveInCase:f<0;g<0}
\begin{align}
\left\vert \omega \right\vert 
& = 
\left\vert 
  m(z)                 \cdot \beta^{-f}
- m(x) \cdot \beta^{f} \cdot \beta^{-f} 
- m(y) \cdot \beta^{g} \cdot \beta^{-f} 
+ m(y) \cdot \beta^{g} \cdot \beta^{-f} 
\right\vert \\
& =
\left\vert 
\left( m(z) - m(x) \cdot \beta^{f} - m(y) \cdot \beta^{g} \right) \cdot \beta^{-f} 
\right\vert
+ 
\left\vert 
m(y) \cdot \beta^{-d} 
\right\vert \\
& \leq
\tfrac{1}{2} \cdot \beta + (M-1),
\label{eqn:muBoundAboveInCase:f<0;g<0->relax}
\end{align}
\end{subequations}
where we have used $f=-1$ and $d \geq 0$ to arrive at \eqref{eqn:muBoundAboveInCase:f<0;g<0->relax}.
So far, we have kept the discussion general with respect to $\beta$. 
Here, we specialize to the case $\beta=2$, which leads to $\vert\omega\vert \leq M$ following \eqref{eqn:muBoundAboveInCase:f<0;g<0->relax}.
If $\vert\omega\vert < M$, then $z - x = \omega \times \beta^{e(x)}$ is a valid representation in $\mathbb D$; 
%
if $\vert\omega\vert = M$, 
then $z - x = \text{sign}(\omega) \cdot \frac{M}{\beta} \times \beta^{e(z)}$ 
is a valid representation in $\mathbb D$. In both cases, we have $w = fl(z - x) = z - x$ as desired.
%
\end{adjustwidth}

We now address the special scenario of \textcolor{kwdcolor}{overflow} when forming $z$ in \eqref{DekkerAdditionEFT3op->z}, i.e., when $\vert x + y \vert > (M-1) \times \beta^{E_{\max}}$. 
If $f<0$, i.e., Cases 3 and 4 above, this cannot happen because 
\[
\vert x+y \vert 
\leq 2 \cdot (M-1) \times \beta^{e(x)} 
\leq (M-1) \times \beta^{e(x)+1} 
\leq (M-1) \times \beta^{e(z)} 
\leq (M-1) \times \beta^{E_{\max}}. 
\]
If $f \geq 0$, i.e., Cases 1 and 2 above, overflow may happen in \eqref{DekkerAdditionEFT3op->z}, which leads to $\vert z \vert = (M - 1) \times \beta^{E_{\max}}$ in Dekker's system.
Since $f = e(x) - e(z) \geq 0$, we have $e(x) = E_{\max}$ as well. 
Moreover, for overflow to happen, $x$ and $y$ must have the same sign.
Hence, $x$ and $z$ must also have the same sign. 
Therefore, we have $\vert m(z) - m(x) \vert < \vert m(z) \vert = M - 1$ and that 
$
z - x = \big(m(z) - m(x) \big) \times \beta^{E_{\max}}
$ 
is a valid representation in Dekker's system.
To sum up, the claim that \eqref{DekkerAdditionEFT3op->w} commits no rounding error still holds when \textcolor{kwdcolor}{overflow} in $z$ is considered.

We proceed to show that \eqref{DekkerAdditionEFT3op->zz} also commits no rounding error under the conditions stated in the theorem.
We first note that for all cases discussed above, including the special case of overflow, $w$ admits a valid representation in $\mathbb D$ with an exponent at least $e(y)$: 
in Case 1, we have $w = y = m(y) \times \beta^{e(y)}$; 
in Case 2, $w$ admits a representation in \eqref{eqn:(z-x)inCase:f>=0;g<0} with exponent $e(z)$, where $e(z) > e(y)$ follows $g > 0$;
Case 3 is a null case;
in Case 4, $w$ admits a representation in \eqref{eqn:(z-x)inCase:f<0;g<0} with exponent $e(x)$, where $e(x) \geq e(y)$ follows $d \geq 0$;
in the special case of overflow, $w$ admits a representation $\big(m(z) - m(x) \big) \times \beta^{E_{\max}}$, where $E_{\max} \geq e(y)$ by definition.
Therefore, we can work with a representation $w = m(w) \times \beta^{e(w)}$, where $e(w) \geq e(y)$.

Introducing $h = e(w) - e(y) \geq 0$, we can write
\[
w = m(w) \cdot \beta^h \times \beta^{e(y)}
\]
%
and
\begin{equation}
\label{eqn:(y-w)inDekkerAdditionEFT3op}
y - w = \left( m(y) - m(w) \cdot \beta^h \right) \times \beta^{e(y)},
\end{equation}
where $m(y) - m(w) \cdot \beta^h \in \mathbb Z$.
Moreover, we have
\begin{subequations}
\label{eqn:(y-w)isExact:boundMantissa}
\begin{align}
\vert y - w \vert 
& = \vert y - (z - x) \vert 
  = \vert (x + y) - z \vert 
\label{eqn:(y-w)isExact:boundMantissa->w=(z-x)} \\
& = \vert (x + y) - fl(x + y) \vert 
\leq \vert (x + y) - x \vert = \vert y \vert,
\label{eqn:(y-w)isExact:boundMantissa->nearestRounding}
\end{align}
\end{subequations}
where \eqref{eqn:(y-w)isExact:boundMantissa->w=(z-x)} used the just proved result that $w = z -x$ and \eqref{eqn:(y-w)isExact:boundMantissa->nearestRounding} used the nearest rounding property. 
%
It follows that $\vert m(y) - m(w) \cdot \beta^h \vert \leq \vert m(y) \vert \leq M - 1$ and, therefore, $y - w \in \mathbb D$ with \eqref{eqn:(y-w)inDekkerAdditionEFT3op} being a valid representation. 
In other words, $zz = fl(y - w) = y - w$. 

\hfill {\bf End of proof for Theorem \ref{thm:3opEFTadditionDekker}.}
\end{theorem}

\begin{remark}[Generality]
\label{rmk:generality}
The statement made in Theorem \ref{thm:3opEFTadditionDekker} is more general than what is typically stated in the literature, e.g., \cite{moller1965quasi, linnainmaa1974analysis,  knuth1997art_vol2, goldberg1991every, higham2002accuracy}, which often imposes the stronger condition that $e(x) \geq e(y)$ holds for the normal representations of $x$ and $y$, or sometimes the even stronger one $\vert x \vert \geq \vert y \vert$.
Instead, Theorem \ref{thm:3opEFTadditionDekker} only requires that $x$ and $y$ admit a combination of representations such that $e(x) \geq e(y)$. They may very well admit representations such that the opposite is true. This generality is exploited in the proof of the following theorem.

\hfill {\bfseries{\itshape End of Remark} \ref*{rmk:generality}.}
\end{remark}

We proceed to show that with three additional arithmetic operations as stated in the following theorem, the condition $e(x) \geq e(y)$ from Theorem \ref{thm:3opEFTadditionDekker} can be removed.
This result has appeared in Dekker's work \cite{dekker1971floating} without a proof. 
The proof given below uses the result of Theorem \ref{thm:3opEFTadditionDekker} and follows a similar pattern.
We assume $\beta=2$ throughout the discussion of Theorem \ref{thm:6opEFTadditionDekker} and its associated lemmas.


\begin{theorem}[6op EFT over addition in Dekker's system]
\label{thm:6opEFTadditionDekker}
Under Rule \ref{Rule:AsIf} and nearest rounding, $\forall x, y \in \mathbb D$, the following formulas achieve EFT $z + zz = x + y:$
\begin{subequations}
\label{DekkerAdditionEFT6op}
\begin{align}
 z & = fl( x  +  y); \label{DekkerAdditionEFT6op->z}  \\
 w & = fl( z  -  x); \label{DekkerAdditionEFT6op->w}  \\
z1 & = fl( y  -  w); \label{DekkerAdditionEFT6op->z1} \\
 v & = fl( w  -  z); \label{DekkerAdditionEFT6op->v}  \\
z2 & = fl( x  +  v); \label{DekkerAdditionEFT6op->z2} \\
zz & = fl( z1 + z2). \label{DekkerAdditionEFT6op->zz}
\end{align}
\end{subequations}
\normalfont
\noindent{\bf Proof.}
The case when $x$ and $y$ admit representations with $e(x) \geq e(y)$ has already been covered by the proof of Theorem \ref{thm:3opEFTadditionDekker} since it is easy to show that $z2 = 0$ following $w = z - x$ in this case.
Therefore, we limit our discussion below to the complementary case, i.e., when $x$ and $y$ {\it do not} admit representations such that $e(x) \geq e(y)$.

In this scenario, aside from the obvious condition $d = e(x) - e(y) < 0$, we also have that $y$ cannot admit a representation with the exponent $E_{\min}$ since, otherwise, any representation of $x$ would lead to $e(x) \geq e(y)$ for this representation of $y$.
Consequently, $y$ admits a normal representation in $\mathbb D$.
%
We work with this normal representation of $y$ in the following.
In other words, when writing
$
y = m(y) \times \beta^{e(y)}
$
below, the conditions $e(y) > E_{\min}$ and $\beta^{t-1} \leq \vert m(y) \vert < \beta^t$ are implied.
These conditions appear in Lemma \ref{lemma:numbersWithinInterval[-1/2,1/2]:e(y)>e(x)ANDyBeingNormal} and Lemma \ref{lemma:numbersWithinInterval[-1,1]:e(y)>e(x)ANDyBeingNormal} invoked in this proof.

The proof below consists of three parts: 
first, we show that $z1=y-w$ under the above conditions; 
second, we show that $w-z2 = z-x$ by invoking Theorem \ref{thm:3opEFTadditionDekker}; 
third, we show that $zz = z1 + z2$.

We proceed to show that $z1=y-w$ still holds 
by separating the discussion into four cases based on the signs of $f$ and $g$, similar to that in the proof of Theorem \ref{thm:3opEFTadditionDekker}.
%
We assume that overflow does not happen when forming $z$ for the four cases below. 
The special scenario of overflow is addressed immediately after the discussion of these cases.

\begin{adjustwidth}{1.5em}{}
\setlength{\parindent}{0em}
\textcolor{kwdcolor}{Case 1: if $f \geq 0$ and $g > 0$}, we again have $\zeta \in \mathbb Z$ as in Case 1 of Theorem \ref{thm:3opEFTadditionDekker}, where $\zeta$ is defined in \eqref{zetaDefinition}, and thus, $z = x + y$. This leads to $w = z - x = y$ 
and $z1 = fl(y - w) = 0 = y - w$ as desired. 
Additionally, we also have $v = -x$, $z2 = 0$, and $zz = 0$ following $w = z - x$.

\textcolor{kwdcolor}{Case 2: if $f \geq 0$ and $g \leq 0$}, we have $d = f - g \geq 0$, which contradicts the assumption $d < 0$. Thus, this is a null case.

\textcolor{kwdcolor}{Case 3: if $f < 0$ and $g \geq 0$}, recall that $z = m(z) \times \beta^{e(z)} = round(\zeta) \times \beta^{e(z)}$, where $\zeta$ is defined in \eqref{zetaDefinition}, 
we have 
\begin{subequations}
\label{eqn:rangeOf(z-x)in6op:case:f<0;g>=0}
\begin{align}
z - x 
& = \left( m(z) - m(x) \cdot \beta^f \right) \times \beta^{e(z)} 
\label{eqn:rangeOf(z-x)in6op:case:f<0;g>=0->a} \\
& = \left( m(z) - m(x) \cdot \beta^f - m(y) \cdot \beta^g + m(y) \cdot \beta^g \right) \times \beta^{e(z)} 
\label{eqn:rangeOf(z-x)in6op:case:f<0;g>=0->b} \\
& = y + \left( m(z) - m(x) \cdot \beta^f - m(y) \cdot \beta^g \right) \times \beta^{e(z)} 
\label{eqn:rangeOf(z-x)in6op:case:f<0;g>=0->c} \\
& \in \left[ y - \tfrac{1}{2} \times \beta^{e(z)} , y + \tfrac{1}{2} \times \beta^{e(z)} \right] 
\subseteq \left[ y - \tfrac{1}{2} \times \beta^{e(y)} , y + \tfrac{1}{2} \times \beta^{e(y)} \right] 
\label{eqn:rangeOf(z-x)in6op:case:f<0;g>=0->d}.
\end{align}
\end{subequations}
From Lemma \ref{lemma:numbersWithinInterval[-1/2,1/2]:e(y)>e(x)ANDyBeingNormal} presented after this theorem, we know 
that $y - fl(z-x) = y-w$ can only take values $\pm 1 \times \beta^{e(y)-1}$, $0$, or $\pm 1 \times \beta^{e(y)}$, all of which are in $\mathbb D$. 
Therefore, we have $z1 = y-w$ as desired.

\textcolor{kwdcolor}{Case 4: if $f < 0$ and $g < 0$}, we have that $z$ admits a representation with $e(z) \leq e(y) + 1$ following Lemma \ref{lemma:e(z)BoundFromAbove}. 
Working with a representation of $z$ such that $e(z) \leq e(y) + 1$ holds, the condition that $g < 0$ further 
implies that $e(z) = e(y) + 1$. 
Similar to \eqref{eqn:rangeOf(z-x)in6op:case:f<0;g>=0}, we have in this case
\begin{equation}
\label{eqn:rangeOf(z-x)in6op:case:f<0;g<0}
z - x 
\in \left[ y - \tfrac{1}{2} \times \beta^{e(z)} , y + \tfrac{1}{2} \times \beta^{e(z)} \right] 
= \left[ y - 1 \times \beta^{e(y)} , y + 1 \times \beta^{e(y)} \right], 
\end{equation}
where $\beta = 2$ is used for the last equality.
When $t > 1$, we can invoke Lemma \ref{lemma:numbersWithinInterval[-1,1]:e(y)>e(x)ANDyBeingNormal} presented after this theorem to arrive at the desired result $z1 = y - w$ as in Case 3.
The special situation $t = 1$ is addressed in Remark \ref{rmk:6opEFT_Case4_t=1} following this theorem.
\end{adjustwidth}

We now address the special scenario of \textcolor{kwdcolor}{overflow} when forming $z$. 
Notice that overflow cannot possibly happen in Cases 1 and 2 above because we have $e(z) < e(y)$ under the combination of assumptions $d < 0$ and $f \geq 0$.
Therefore, we need only to concentrate on Cases 3 and 4 above, where $f < 0$.
%
%
Moreover, since $x$ and $y$ must have the same sign for $z$ to overflow, 
it is sufficient to address only the case where both $x$ and $y$ are positive. 
Because overflow leads to clipping in Dekker's system, we have $z = \big(\beta^t - 1\big) \times \beta^{E_{\max}}$ and, therefore,
\[
z - x 
= \left( \big(\beta^t-1\big) - m(x) \cdot \beta^{f} \right) \times \beta^{E_{\max}} 
\geq \left( \big(\beta^t-1\big) - \big(\beta^t-1\big) \cdot \beta^{-1} \right) \times \beta^{E_{\max}} 
= \big(\beta^t-1\big) \times \beta^{E_{\max}-1},
\]
where $f \leq -1$ is used for the inequality and $\beta = 2$ is used for the last equality.
Moreover, because of overflow, we also have $x+y>z$, i.e., $z-x<y$.
Thus, by monotonicity, we have 
\begin{equation}
\label{eqn:bounds_of_w_in_6op:case:overflow}
\big(\beta^t-1\big) \times \beta^{E_{\max}-1} \leq w \leq y. 
\end{equation}
Because of \eqref{eqn:bounds_of_w_in_6op:case:overflow}, we can write 
\begin{subequations}
\label{eqn:y_and_w:6op:case:overflow}
\begin{align}
y & = \tilde{m}(y) \times \beta^{E_{\max}-1} \label{eqn:y:6op:case:overflow}; \\
w & = \tilde{m}(w) \times \beta^{E_{\max}-1} \label{eqn:w:6op:case:overflow},
\end{align}
\end{subequations}
where $\tilde{m}(y)\in\mathbb Z$ and $\tilde{m}(w)\in\mathbb Z$.
Notice that we have used $\tilde{m}(\cdot)$ instead of $m(\cdot)$ for the mantissas in \eqref{eqn:y:6op:case:overflow} and \eqref{eqn:w:6op:case:overflow} because they are not necessarily valid representations in $\mathbb D$.
Moreover, since $\tilde{m}(y) \leq \big(\beta^t - 1\big) \cdot \beta$ and $\tilde{m}(w) \geq \big(\beta^t-1\big)$, we can write
\begin{equation}
\label{eqn:y-w:6op:case:overflow}
y - w = \big(\tilde{m}(y) - \tilde{m}(w)\big) \times \beta^{E_{\max}-1},
\end{equation}
where $\tilde{m}(y) - \tilde{m}(w) \in \mathbb Z$ and $0 \leq \tilde{m}(y) - \tilde{m}(w) \leq \beta^t - 1$.
Thus, $y-w \in \mathbb D$ with \eqref{eqn:y-w:6op:case:overflow} being a valid representation, which leads to $z1 = y - w$ as desired.
This concludes the discussion for the special scenario of \textcolor{kwdcolor}{overflow} when forming $z$.

Now that we completed the demonstration of $z1 = y - w$, we proceed to show that $w-z2 = z-x$.
When $f \geq 0$, from the discussion above, Case 1 is the only nontrivial case under the conditions of the theorem, 
for which $w-z2 = z-x$ holds because $w = z - x$ and $z2 = 0$.
When $f<0$, i.e., $e(x) < e(z)$, following Theorem \eqref{thm:3opEFTadditionDekker}, operations \eqref{DekkerAdditionEFT6op->w}, \eqref{DekkerAdditionEFT6op->v}, and \eqref{DekkerAdditionEFT6op->z2} form a 3op EFT over addition (after adjusting for the signs), which leads to $w - z2 = z - x$.
Combined with $z1 = y - w$ obtained earlier, this leads to 
\begin{equation}
\label{eqn:(z+z1+z2)in6op}
z + z1 + z2 = x + y.
\end{equation}

It remains to show that $zz = z1 + z2$, i.e., \eqref{DekkerAdditionEFT6op->zz} also commits no rounding error.
For all non-empty cases discussed above, including the special scenario of overflow when forming $z$ and the special systems with $t=1$, we have that $z1 = y - w$ admits a representation in $\mathbb D$ with an exponent no less than $e(y) - 1$. 
Since $d = e(x) - e(y) < 0$, it follows that $z1$ can be expressed as $\tilde{m}(z1) \times \beta^{e(x)}$ where $\tilde{m}(z1) \in \mathbb Z$.
%
Moreover, from the proof of Theorem \eqref{thm:3opEFTadditionDekker}, we have that $z2$ can be expressed as $m(z2) \times \beta^{e(x)}$ where $m(z2) \in \mathbb Z$ and $\vert m(z2) \vert < M$.
%
Therefore, $z1+z2$ can be expressed as an integer times $\beta^{e(x)}$ as well, i.e., $z1+z2 = \Big(\tilde{m}(z1) + m(z2)\Big) \times \beta^{e(x)}$.
From \eqref{eqn:(z+z1+z2)in6op}, we have
\[
\vert z1 + z2 \vert = \vert x + y - z \vert \leq \vert x + y - y \vert = \vert x \vert
\]
following the nearest rounding property, 
which leads to $\vert \tilde{m}(z1) + m(z2) \vert \leq \vert m(x) \vert < M$.
Therefore, $z1+z2 \in \mathbb D$ 
and 
$zz = z1 + z2$ follows.
%

\hfill {\bf End of proof for Theorem \ref{thm:6opEFTadditionDekker}.}
\end{theorem}

\begin{remark}[Special scenario: $t=1$] 
\label{rmk:6opEFT_Case4_t=1}
We use this remark to address the special scenario when $t=1$ in Case 4 of Theorem \ref{thm:6opEFTadditionDekker}. 
When $t=1$, the mantissa can only take value $0$ or $\pm 1$. 
If $z = 0$, $z$ obviously admits a representation with exponent $E_{\min}$ and,
hence, we can invoke Lemma \ref{lemma:additionIsExactWhenEmin} to arrive at $z = x + y$.
If either $x$ or $y$ is $0$, $z = x + y$ also holds since $x,y \in \mathbb D$.
Following the same argument as in Case 1 of Theorem \ref{thm:6opEFTadditionDekker}, we arrive at $zz = 0$ and $z + zz = x + y$ subsequently.

Given the above observation, we assume that $x$, $y$, and $z$ all have nonzero mantissa below.
Recall the conditions $f = e(x) - e(z) < 0$ and $g = e(y) - e(z) < 0$ for Case 4 of Theorem \ref{thm:6opEFTadditionDekker}, combined with $t=1$, we have that $\vert x \vert < \vert z \vert$ and $\vert y \vert < \vert z \vert$, which imply that $x$ and $y$ must have the same sign, and so does $z$. 
Therefore, it suffices to consider the case where $x$, $y$, and $z$ all have mantissa $1$.

Recall that $z$ admits a representation with $e(z) = e(y) + 1$ from Case 4 of Theorem \ref{thm:6opEFTadditionDekker}, we can write $y = 1 \times \beta^{e(y)}$ and $z = 1 \times \beta^{e(y)+1}$.
%
Moreover, following the condition $d = e(x) - e(y) < 0$, it is necessary that $e(x) = e(y) - 1$ and $x = 1 \times \beta^{e(y) - 1}$ because, otherwise, i.e., if $e(x) < e(y) - 1$, $x + y$ would round to $y$ instead of $z$.

Therefore, we have $z-x = (1 + \nicefrac{1}{2}) \times \beta^{e(y)}$, which means that $w = fl(z-x)$ can be either $1 \times \beta^{e(y)}$ or $1 \times \beta^{e(y)+1}$, depending on the tie-breaking rule. 
Therefore, $y-w$ can only take value $0$ or $-1 \times \beta^{e(y)}$, both of which are in $\mathbb D$. 
%
Hence, we arrive at the desired result $z1 = y - w$.

\hfill {\bfseries{\itshape End of Remark} \ref*{rmk:6opEFT_Case4_t=1}.}
\end{remark}

The following lemma is invoked for Case 3 in the proof of Theorem \ref{thm:6opEFTadditionDekker}.
\begin{lemma}
\label{lemma:numbersWithinInterval[-1/2,1/2]:e(y)>e(x)ANDyBeingNormal}
With $\beta=2$ and nearest rounding, given 
$y = m(y) \times \beta^{e(y)} \in \mathbb D$ 
with $\beta^{t-1} \leq \vert m(y) \vert < \beta^t$ and $e(y) > E_{\min}$, 
the rounding outcome of a real number 
$\tilde{w} \in \left[ y - \tfrac{1}{2} \times \beta^{e(y)} , 
                   y + \tfrac{1}{2} \times \beta^{e(y)} \right]$ 
admits a normal representation with exponent $e(y)-1$, $e(y)$, or $e(y)+1$.
Moreover, the difference $y-fl(\omega)$ can take values $\pm 1 \times \beta^{e(y)-1}$, $0$, or $\pm 1 \times \beta^{e(y)}$ only.
\\
\normalfont
{\bf Proof.}
We have $y \neq 0$ since $\vert m(y) \vert \geq \beta^{t-1}$. 
Moreover, it suffices to consider only the case $y>0$ because of symmetry.

We first address the special case $t=1$, where the mantissa can only take value 1.
We can relax the interval to 
$\left[ y - \tfrac{1}{2} \times \beta^{e(y)} , 
        y + 1            \times \beta^{e(y)} \right]$,
which becomes $\left[ 1 \times \beta^{e(y)-1} , 1 \times \beta^{e(y)+1} \right]$ in this case.
Since both endpoints of the above interval are in $\mathbb D$, 
by Lemma \ref{lemma:Monotonicity} (Monotonicity), 
we have that $\tilde{w}$ can only round to one of the following three floating point numbers:
\begin{equation}
\label{eqn:numbersOmegaCanRoundto:[-1/2,1/2]:t=1}
\left[ 1 \times \beta^{e(y)-1} , 1 \times \beta^{e(y)+1} \right] \intersect \mathbb D
=
\left\{ 
1 \times \beta^{e(y)-1}, 
1 \times \beta^{e(y)}, 
1 \times \beta^{e(y)+1} 
\right\}.
\end{equation}
Consequently, the difference $y - fl(\tilde{w})$ can only take values 
\begin{equation}
\label{eqn:numbers(y-w)canTake:[-1/2,1/2]:t=1}
\left\{ 
  1 \times \beta^{e(y)-1}, \enskip
  0, \enskip
- 1 \times \beta^{e(y)} 
\right\}.
\end{equation}
{\it We note here that the last member of \eqref{eqn:numbersOmegaCanRoundto:[-1/2,1/2]:t=1}, i.e., $1 \times \beta^{e(y)+1}$, may overflow. In that case, $\tilde{w}$ can only round to the first two members of \eqref{eqn:numbersOmegaCanRoundto:[-1/2,1/2]:t=1} in Dekker's system.}

With the above special treatment for $t=1$, we assume $t>1$ and separate the discussion below based on the size of $m(y)$.

\begin{adjustwidth}{1.5em}{}
\setlength{\parindent}{0em}
\textcolor{kwdcolor}{Case 1: if $m(y) = \beta^{t-1}$}, we can relax the interval to 
$\left[ y - \tfrac{1}{2} \times \beta^{e(y)} , 
        y + 1            \times \beta^{e(y)} \right]$, 
which becomes
$\left[ \big(\beta^t     - 1 \big) \times \beta^{e(y)-1} , 
        \big(\beta^{t-1} + 1 \big) \times \beta^{e(y)} \right]$ in this case, 
with both endpoints in $\mathbb D$. 
By monotonicity, $\tilde{w}$ can only round to one of the following three floating point numbers:
\begin{equation}
\label{eqn:numbersOmegaCanRoundto:[-1/2,1/2]:m(y)isSmallest}
\left[ \big(\beta^t     - 1 \big) \times \beta^{e(y)-1} , 
       \big(\beta^{t-1} + 1 \big) \times \beta^{e(y)} \right] \intersect \mathbb D
=
\left\{ 
\big(\beta^t - 1 \big) \times \beta^{e(y)-1}, 
\beta^{t-1} \times \beta^{e(y)}, 
\big(\beta^{t-1} + 1 \big) \times \beta^{e(y)} 
\right\}.
\end{equation}
Consequently, the difference $y - fl(\tilde{w})$ can only take values
\begin{equation}
\label{eqn:numbers(y-w)canTake:[-1/2,1/2]:m(y)isSmallest}
\left\{ 
  1 \times \beta^{e(y)-1}, \enskip
  0, \enskip
- 1 \times \beta^{e(y)}
\right\}.
\end{equation}

\textcolor{kwdcolor}{Case 2: if $\beta^{t-1} < m(y) < \beta^{t} - 1$}, we can relax the interval to 
$\left[ y - 1 \times \beta^{e(y)} , 
        y + 1 \times \beta^{e(y)} \right]$, 
which becomes 
$\left[ \big( m(y) - 1 \big) \times \beta^{e(y)} , 
        \big( m(y) + 1 \big) \times \beta^{e(y)} \right]$ in this case, 
with both endpoints in $\mathbb D$. 
By monotonicity, $\tilde{w}$ can only round to one of the following three floating point numbers: 
\begin{equation}
\label{eqn:numbersOmegaCanRoundto:[-1/2,1/2]:m(y)isInTheInterior}
\left[ \big( m(y) - 1 \big) \times \beta^{e(y)} , 
       \big( m(y) + 1 \big) \times \beta^{e(y)} \right] \intersect \mathbb D
=
\left\{ 
\big( m(y) - 1 \big) \times \beta^{e(y)}, 
      m(y)           \times \beta^{e(y)}, 
\big( m(y) + 1 \big) \times \beta^{e(y)} 
\right\}.
\end{equation}
Consequently, the difference $y - fl(\tilde{w})$ can only take values
\begin{equation}
\label{eqn:numbers(y-w)canTake:[-1/2,1/2]:m(y)isInTheInterior}
\left\{ 
  1 \times \beta^{e(y)}, \enskip
  0, \enskip
- 1 \times \beta^{e(y)}
\right\}.
\end{equation}
{\it We note here that if $t=1$ or $2$, we would have $\beta^{t-1} + 1 \geq \beta^{t} - 1$, leading to a null case here.
}

\textcolor{kwdcolor}{Case 3: if $m(y) = \beta^{t} - 1$}, we can relax the interval to 
$\left[ y - 1 \times \beta^{e(y)} , 
        y + 1 \times \beta^{e(y)} \right]$, 
which becomes 
$\left[ \big(\beta^t - 2 \big) \times \beta^{e(y)} , 
                   \beta^{t-1} \times \beta^{e(y)+1} \right]$ in this case, 
with both endpoints in $\mathbb D$. 
By monotonicity, $\tilde{w}$ can only round to one of the following three floating point numbers: 
\begin{equation}
\label{eqn:numbersOmegaCanRoundto:[-1/2,1/2]:m(y)isLargest}
\left[ \big(\beta^t - 2 \big) \times \beta^{e(y)} , 
       \big(\beta^{t-1} \big) \times \beta^{e(y)+1} \right] \intersect \mathbb D 
= 
\left\{ 
\big(\beta^t - 2 \big) \times \beta^{e(y)}, 
\big(\beta^t - 1 \big) \times \beta^{e(y)}, 
     \beta^{t-1}       \times \beta^{e(y)+1} 
\right\}.
\end{equation}
Consequently, the difference $y - fl(\tilde{w})$ can only take values
\begin{equation}
\label{eqn:numbers(y-w)canTake:[-1/2,1/2]:m(y)isLargest}
\left\{ 
  1 \times \beta^{e(y)}, \enskip
  0, \enskip 
- 1 \times \beta^{e(y)} 
\right\}.
\end{equation}
{\it We note here that the last member of \eqref{eqn:numbersOmegaCanRoundto:[-1/2,1/2]:m(y)isLargest}, i.e., $\beta^{t-1} \times \beta^{e(y)+1}$, may overflow. In that case, $\tilde{w}$ can only round to the first two members of \eqref{eqn:numbersOmegaCanRoundto:[-1/2,1/2]:m(y)isLargest} in Dekker's system.}
\end{adjustwidth}

\hfill {\bf End of proof for Lemma \ref{lemma:numbersWithinInterval[-1/2,1/2]:e(y)>e(x)ANDyBeingNormal}.}
\end{lemma}

The following lemma is invoked for Case 4 in the proof of Theorem \ref{thm:6opEFTadditionDekker}. It relaxes the interval to a larger one but excludes the special case where $t=1$. 
\begin{lemma}
\label{lemma:numbersWithinInterval[-1,1]:e(y)>e(x)ANDyBeingNormal}
With $\beta = 2$, $t > 1$, and nearest rounding, 
given $y = m(y) \times \beta^{e(y)} \in \mathbb D$ 
with $\beta^{t-1} \leq \vert m(y) \vert < \beta^t$ and $e(y) > E_{\min}$, 
the rounding outcome of a real number 
$\tilde{w} \in \left[ y - 1 \times \beta^{e(y)} , y + 1 \times \beta^{e(y)} \right]$ admits a normal representation with exponent $e(y)-1$, $e(y)$, or $e(y)+1$. 
Moreover, the difference $y-fl(\tilde{w})$ can take values $\pm 1 \times \beta^{e(y)-1}$, $0$, or $\pm 1 \times \beta^{e(y)}$ only. \\
\normalfont
{\bf Proof.}
The proof mostly follows that of Lemma \ref{lemma:numbersWithinInterval[-1/2,1/2]:e(y)>e(x)ANDyBeingNormal}.
As in Lemma \ref{lemma:numbersWithinInterval[-1/2,1/2]:e(y)>e(x)ANDyBeingNormal}, it suffices to consider only the case $y>0$.
%

\begin{adjustwidth}{1.5em}{}
\setlength{\parindent}{0em}
\textcolor{kwdcolor}{Case 1: if $m(y) = \beta^{t-1}$}, the interval becomes 
$\left[ \big(\beta^t     - \beta \big) \times \beta^{e(y)-1} , 
        \big(\beta^{t-1} + 1     \big) \times \beta^{e(y)} \right]$, 
with both endpoints in $\mathbb D$. 
By monotonicity, $\tilde{w}$ can only round to one of the following $\beta+2$ (i.e., $4$ since $\beta=2$) floating point numbers between the two endpoints:
\begin{equation}
\label{eqn:numbersOmegaCanRoundto:[-1,1]:m(y)isSmallest}
\left\{ 
\big(\beta^t - 2 \big) \times \beta^{e(y)-1}, \enskip
\big(\beta^t - 1 \big) \times \beta^{e(y)-1}, \enskip
\beta^{t-1} \times \beta^{e(y)}, \enskip
\big(\beta^{t-1} + 1 \big) \times \beta^{e(y)} 
\right\}.
\end{equation}
Consequently, the difference $y - fl(\tilde{w})$ can only take values
\begin{equation}
\label{eqn:numbers(y-w)canTake:[-1,1]:m(y)isSmallest}
\left\{ 
  1 \times \beta^{e(y)}, \enskip
  1 \times \beta^{e(y)-1}, \enskip
  0, \enskip 
- 1 \times \beta^{e(y)} 
\right\}.
\end{equation}

\textcolor{kwdcolor}{Case 2: if $\beta^{t-1} < m(y) < \beta^{t} - 1$}, this case has already been covered in the proof of Case 2 of Lemma \ref{lemma:numbersWithinInterval[-1/2,1/2]:e(y)>e(x)ANDyBeingNormal}.

\textcolor{kwdcolor}{Case 3: if $m(y) = \beta^{t} - 1$}, this case has already been covered in the proof of Case 3 of Lemma \ref{lemma:numbersWithinInterval[-1/2,1/2]:e(y)>e(x)ANDyBeingNormal}.
\end{adjustwidth}

\hfill {\bf End of proof for Lemma \ref{lemma:numbersWithinInterval[-1,1]:e(y)>e(x)ANDyBeingNormal}.}
\end{lemma}

Before proceeding, we introduce the notation $\epsilon$ for the quantity {\it unit round-off}.
Given a floating point number system with $t$-digit mantissa, $\epsilon$ is defined as $\epsilon = \tfrac{1}{2}\cdot\beta^{1-t}$ under nearest rounding. 
The following lemma regarding $\epsilon$ can be established.
%


\begin{lemma}
\label{lemma:bound_on_rounding_with_rounding_output} 
For $\tilde{z} \in \mathbb R$ that rounds to $z = m(z) \times \beta^{e(z)}$ in $\mathbb D$, if $\vert \tilde{z} \vert \leq (M-1)\times\beta^{E{\max}}$, i.e., $\tilde{z}$ does not overflow, we have $\vert z - \tilde{z} \vert \leq \tfrac{1}{2} \times \beta^{e(z)}$.
Moreover, when $z$ is not a subnormal number, 
we also have $\vert z - \tilde{z} \vert \leq \epsilon \vert z \vert$.
%
\\
\normalfont
{\bf Proof.}
Expressing $\tilde{z}$ as $\tilde{z} = \zeta \times \beta^{e(z)}$, we have $m(z) = round(\zeta)$ following Lemma \ref{lemma:relationBetweenFloatingPointRoundingAndMantissaIntegerRounding} and, therefore, 
$\vert m(z) - \zeta \vert \leq \tfrac{1}{2}$.
It follows that 
$\vert z - \tilde{z} \vert = 
\vert m(z) - \zeta \vert \times \beta^{e(z)} \leq 
\tfrac{1}{2} \times \beta^{e(z)}$.

Notice that the result above 
holds for any valid representation of $z$. 
When $z$ is not a subnormal number, working with its normal representation $m(z) \times \beta^{e(z)}$ where $\vert m(z) \vert \geq \beta^{t-1}$, we have 
$\vert z - \tilde{z} \vert \leq 
\tfrac{1}{2} \times \beta^{e(z)} = 
\tfrac{1}{2} \cdot \beta^{1-t} \cdot \beta^{t-1} \times \beta^{e(z)} \leq 
\epsilon \vert z \vert
$.

\hfill {\bf End of proof for Lemma \ref*{lemma:bound_on_rounding_with_rounding_output}.}
\end{lemma}

Lemma \ref{lemma:bound_on_rounding_with_rounding_output} provides a statement regarding rounding in general.
If the number to be rounded, i.e., $\tilde{z}$, is the intermediate result from an addition (or subtraction) operation, the additional requirement that {\it $z$ is not a subnormal number} can be dropped from the statement, which is summarized below.

\begin{lemma}
\label{lemma:bound_on_rounding_with_rounding_output_addition_special} 
Given $x,y \in \mathbb D$ and $\tilde{z} = x + y \in \mathbb R$, which rounds to $z = m(z) \times \beta^{e(z)}$ in $\mathbb D$, 
if $\vert \tilde{z} \vert \leq (M-1)\times\beta^{E{\max}}$, i.e., $\tilde{z}$ does not overflow, we have $\vert z - \tilde{z} \vert \leq \tfrac{1}{2} \times \beta^{e(z)}$ and $\vert z - \tilde{z} \vert \leq \epsilon \vert z \vert$.
\\
\normalfont
{\bf Proof.}
It suffices to consider only the case when $z$ is a subnormal number. 
In this case, we have $z = x + y = \tilde{z}$ from Lemma \ref{lemma:additionIsExactWhenEmin} and, hence, the stated results still hold.

\hfill {\bf End of proof for Lemma \ref*{lemma:bound_on_rounding_with_rounding_output_addition_special}.}
\end{lemma}

The following corollary results from a straightforward application of Lemma \ref{lemma:bound_on_rounding_with_rounding_output_addition_special} by observing that $zz = x+y - z = x+y - fl(x+y)$ when the conditions of Theorem \ref{thm:3opEFTadditionDekker} or Theorem \ref{thm:6opEFTadditionDekker} are met.
\begin{corollary}
\label{col:bound_zz_by_exponent_e(z)}
Under the conditions of Theorem \ref{thm:3opEFTadditionDekker} or the conditions of Theorem \ref{thm:6opEFTadditionDekker} and assuming overflow does not happen when forming $z$, we have $\vert zz \vert \leq \tfrac{1}{2} \times \beta^{e(z)}$ and $\vert zz \vert \leq \epsilon \vert z \vert$.
\end{corollary}

Lemma \ref{lemma:bound_on_rounding_with_rounding_output_addition_special} 
bounds the error with the {\it output} of the floating point addition operation.
The lemma below bounds the error with the {\it inputs} of the addition operation.
\begin{lemma}
\label{lemma:bound_on_rounding_with_rounding_inputs} 
Given $x,y \in \mathbb D$ and $\tilde{z} = x + y \in \mathbb R$, which rounds to $z$, 
if $\vert \tilde{z} \vert \leq (M-1)\times\beta^{E{\max}}$, i.e., $\tilde{z}$ does not overflow, we have $\vert z - \tilde{z} \vert \leq \epsilon \vert \tilde{z} \vert$.
\\
\normalfont
{\bf Proof.} It suffices to consider only the case $\tilde{z} \geq 0$ since both $\mathbb D$ and $\mathbb R$ are symmetric. The two cases discussed below cover all possible outcomes of $x+y$ that do not lead to overflow. 

When $\tilde{z} \in \left[ 0 , \beta^t \times \beta^{E_{\min}} \right]$, we can express $\tilde{z}$ as $\zeta \times \beta^{E_{\min}}$, where $\zeta = m(x) \cdot \beta^{e(x)-E_{\min}} + m(y) \cdot \beta^{e(y)-E_{\min}}$.
Since $e(x),e(y) \geq E_{\min}$, we have that $\zeta \in \mathbb Z$ and hence can take values $0,1,2,\cdots,\beta^t$. Since overflow has been precluded, these possible options of $\zeta$ all lead to $\tilde{z} \in \mathbb D$ and hence $z - \tilde{z} = 0$.

When $\tilde{z} \in \left[ \beta^{t-1} \times \beta^{E} , \beta^t \times \beta^{E} \right]$ for $E_{\min} < E \leq E_{\max}$, 
noticing that the space between two adjacent floating point numbers in this interval is $1 \times \beta^E$ and since overflow has been precluded, 
we have $\vert z - \tilde{z} \vert \leq \tfrac{1}{2} \times \beta^E$ following nearest rounding and, therefore, 
$\vert z - \tilde{z} \vert \leq \tfrac{1}{2} \cdot \beta^{1-t} \cdot \beta^{t-1} \times \beta^E \leq \epsilon \vert \tilde{z} \vert$. 
%
(In this case, we actually have a strict inequality 
$\vert z - \tilde{z} \vert < \epsilon \vert \tilde{z} \vert$ because $z - \tilde{z}$ would be zero if $\tilde{z} = \beta^{t-1} \times \beta^{E}$.)

\hfill {\bf End of proof for Lemma \ref*{lemma:bound_on_rounding_with_rounding_inputs}.}
\end{lemma}

The following corollary results from a straightforward application of Lemma \ref{lemma:bound_on_rounding_with_rounding_inputs} by observing that $zz = x+y - z = x+y - fl(x+y)$ when the conditions of Theorem \ref{thm:3opEFTadditionDekker} or Theorem \ref{thm:6opEFTadditionDekker} are met.

\begin{corollary}
\label{col:bound_zz_by_(x+y)}
Under the conditions of Theorem \ref{thm:3opEFTadditionDekker} or the conditions of Theorem \ref{thm:6opEFTadditionDekker} and assuming overflow does not happen when forming $z$, we have $\vert zz \vert \leq \epsilon \vert \tilde{z} \vert$ where $\tilde{z} = x+y$.
\end{corollary}

\section{Application of the EFTs in summing a sequence of numbers}
\label{sec:compensatedSum}
The EFTs presented in Section \ref{sec:EFTunderAddition} can bring benefits in various applications; see \cite{quinlan1994round, shewchuk1997adaptive, chowdhary2022fast} for a few examples. 
One such application is summing a sequence of numbers, in particular, summing a sequence of numbers in the order that they are presented, which is referred to as recursive summation in Higham's work \cite{higham1993accuracy}; see also \cite[P.~80]{higham2002accuracy}.
Recursive summation is disadvantaged in accuracy when compared to the other summing strategies examined in \cite{higham1993accuracy} that require the control over the summing order.
However, for many applications, particularly those involving iterative procedures where the addends are produced on the fly, recursive summation is the only viable option.
%

The EFT from Theorem \ref{thm:3opEFTadditionDekker} can be used to improve the accuracy of recursive summation. 
This has been recognized in the early work of Kahan \cite{kahan1965pracniques}, which presented a technique often referred to as Kahan's summation or Kahan's compensated summation.
A proof about the error bound of Kahan's summation can be found in \cite[P.~244]{knuth1997art_vol2} and \cite{goldberg1991every}.
In the following, we will start with formalizing the recursive summation problem in Section \ref{sumSequenceFormulation} and summarizing Kahan's compensated summation algorithm in Section \ref{sec:sumSequence3op}. A natural extension of Kahan's summation is examined in Section \ref{sec:sumSequence6op}.

More nuanced versions of compensated summation algorithms are then presented and analyzed in Sections \ref{sec:sumSequenceDouble6op} and \ref{sec:sumSequenceTriple6op}. 
These more nuanced versions 
enjoy tighter error bounds, 
which can be particularly beneficial for designing validation tests for computer systems.
%
Thanks to the adoption of Dekker's system and the theoretic results developed above, 
the proofs provided below are simple to follow and 
cover 
subnormal numbers naturally.
For the remainder of this section, we assume nearest rounding and no overflow.

\subsection{Formulation of the problem}
\label{sumSequenceFormulation}
We want to carry out the calculation of the mathematical expression $\tilde{s} = x_1 + x_2 + \cdots x_n$ on a computer, where $\tilde{s}$ denotes the exact outcome of the expression if it is carried out in the real number system. 
All addends $x_1 \cdots x_n$ are assumed to be representable in the floating point number system.
We do {\it not} assume any special relationship among the signs or magnitudes of the addends.
The procedure of recursive summation is illustrated in Algorithm \ref{alg:recursive_summation}.

\renewcommand{\lstlistingname}{Algorithm}
\renewcommand{\lstlistlistingname}{List of \lstlistingname s}

\begin{center}
\begin{minipage}{0.6\linewidth}
\centering
\captionsetup[lstlisting]{font=small}
\begin{lstlisting}[
tabsize=4,
basicstyle=\ttfamily\small,
xleftmargin=8em,
caption={recursive summation.},
label={alg:recursive_summation},
]
s = 0;
for i = 1:n
    s = s + x[i];
end
\end{lstlisting}
\end{minipage}
\end{center}
In Algorithm \ref{alg:recursive_summation}, the notation \texttt{1:n} denotes a sequence from \texttt{1} to \texttt{n} with increment \texttt{1}.
Error analysis for the recursive summation is straightforward and can be found in, e.g., \cite{higham2002accuracy}.
Using $s_i$ to denote the running sum at each iteration, we have, following Lemma \ref{lemma:bound_on_rounding_with_rounding_inputs}, 
\begin{equation}
\label{runningExpressionsRecursiveSummation}
\begin{array}{rcl}
s_1 & = & x_1(1+\delta_1) \\
s_2 & = & x_1(1+\delta_1)(1+\delta_2) + x_2(1+\delta_2) \\
& \vdots & \\
s_n & = & x_1 \prod_{i=1}^{n}(1+\delta_i)
        + x_2 \prod_{i=2}^{n}(1+\delta_i)
        + \cdots 
        + x_n \prod_{i=n}^{n}(1+\delta_i),
\end{array}
\end{equation}
where $\vert \delta_i \vert \leq \epsilon$ is the relative error committed at iteration $i$ and $s_n$ is the output of Algorithm \ref{alg:recursive_summation}.

\begin{remark}[Conciseness]
\label{rmk:conciseness_in_error_bounds}
We note here that $\delta_1$ from \eqref{runningExpressionsRecursiveSummation} should be zero for any reasonable machine since we expect the calculation of $(0+x_1)$ to incur {\it no} rounding error.
However, as noted in \cite{higham2002accuracy}, ``$\cdots$we should not attach too much significance to the precise values of the constants in error bounds''
and, from the same work, 
``It is worth spending effort, though, to put error bounds in a concise, easily interpreted form.''
Padding with the extra term $(1+\delta_1)$ here does make the error bounds more concise and easily interpretable.

\hfill {\bfseries{\itshape End of Remark} \ref*{rmk:conciseness_in_error_bounds}.}
\end{remark}

Invoking Lemma 3.1 of \cite{higham2002accuracy}, we arrive at the following error bound:
\begin{equation}
\label{forwardErrorBoundRecursiveSummation}
\vert s_n - \tilde{s} \vert \leq \frac{n\epsilon}{1-n\epsilon} \textstyle \sum_{i=1}^n \vert x_i \vert \,,
\end{equation}
under the assumption $n\epsilon < 1$.
Two 
observations can be made on this error bound. First, assuming $n\epsilon \ll 1$, the error bound grows approximately linearly with $n$, the number of addends being summed. Second, \eqref{forwardErrorBoundRecursiveSummation} does not provide a sharp bound on the relative error when $\sum_{i=1}^n \vert x_i \vert \gg \left\vert \sum_{i=1}^n x_i \right\vert$.

\subsection{Recursive summation with 3op compensation}
\label{sec:sumSequence3op}
Kahan's summation method improves the error 
by using the 3op EFT to recover the error committed at each iteration and {\it compensating} for it at the next iteration. It is summarized in Algorithm \ref{alg:recursive_summation_3op}, where \texttt{add\_3op} implements the procedure from Theorem \ref{thm:3opEFTadditionDekker}.

\begin{center}
\begin{minipage}{0.6\linewidth}
\centering
\captionsetup[lstlisting]{font=small}
\begin{lstlisting}[
tabsize=4,
basicstyle=\ttfamily\small,
xleftmargin=6.5em,
caption={recursive summation with 3op compensation.},
label={alg:recursive_summation_3op},
]
s = 0;
e = 0;
for i = 1:n
    y = e + x[i];
    [s,e] = add_3op(s,y);
end
\end{lstlisting}
\end{minipage}
\end{center}
The following error bound for Algorithm \ref{alg:recursive_summation_3op} is given in \cite[P.~244]{knuth1997art_vol2} and \cite{goldberg1991every} by tracking the leading order term:
\begin{equation}
\label{forwardErrorBoundRecursiveSummation3op}
\vert s_n - \tilde{s} \vert 
\leq 
\left(2\epsilon + \mathcal{O}(n\epsilon^2) \right) \textstyle \sum_{i=1}^{n} \vert x_i \vert.
\end{equation}
We note here that compared to the bound given in \eqref{forwardErrorBoundRecursiveSummation}, this error bound does {\it not} grow linearly with $n$ so long as the high order terms are negligible, but still does not bound the relative error sharply when $\sum_{i=1}^n \vert x_i \vert \gg \left\vert \sum_{i=1}^n x_i \right\vert$.

\subsection{Recursive summation with 6op compensation}
\label{sec:sumSequence6op}

The procedure \texttt{add\_3op} in Algorithm \ref{alg:recursive_summation_3op} can be replaced by the 6op version of the EFT presented in Theorem \ref{thm:6opEFTadditionDekker}, resulting in Algorithm \ref{alg:recursive_summation_6op} below.
\begin{center}
\begin{minipage}{0.6\linewidth}
\centering
\captionsetup[lstlisting]{font=small}
\begin{lstlisting}[
tabsize=4,
basicstyle=\ttfamily\small,
xleftmargin=6.5em,
caption={recursive summation with 6op compensation.},
label={alg:recursive_summation_6op},
]
s = 0;
e = 0;
for i = 1:n
    y = e + x[i];
    [s,e] = add_6op(s,y);
end
\end{lstlisting}
\end{minipage}
\end{center}
Attaching subscript index $i$ to disambiguate variables from different iterations,
we can write
\begin{subequations}
\label{eqn:recursive_summation_6op_perturbations}
\begin{align}
y_i & = (1+\alpha_i) (e_{i-1} + x_i) \label{eqn:recursive_summation_6op_perturbations_y} \\
s_i & = (1+\beta_i ) (s_{i-1} + y_i) \label{eqn:recursive_summation_6op_perturbations_s} \\
e_i & =  - \beta_i   (s_{i-1} + y_i) \label{eqn:recursive_summation_6op_perturbations_e} 
\end{align}
\end{subequations}
with $\vert \alpha_i \vert, \vert \beta_i \vert \leq \epsilon$ following Lemma \ref{lemma:bound_on_rounding_with_rounding_inputs}.
The analysis for Algorithm \ref{alg:recursive_summation_6op} is considerably simpler and cleaner than that for Algorithm \ref{alg:recursive_summation_3op} because the EFT relation $s_i + e_i = s_{i-1} + y_i$ always holds.
The following recursive formula can be derived from \eqref{eqn:recursive_summation_6op_perturbations}
\begin{equation}
\label{eqn:recursive_summation_6op_recursive_relation}
s_i + e_i 
= s_{i-1} + y_i \\
= (1 - \alpha_i \beta_{i-1}) ( s_{i-1} + e_{i-1} ) + (1 + \alpha_i) x_i \, ,
\end{equation}
where we have used $e_{i-1} = - \beta_{i-1} (s_{i-1} + e_{i-1})$ in the derivation.
Applying \eqref{eqn:recursive_summation_6op_recursive_relation} recursively and noticing that $s_0 = e_0 = 0$, we arrive at
\begin{equation}
\label{eqn:recursive_summation_6op_full_characterization}
\textstyle
s_n + e_n = \sum_{k=1}^n \left\{ \prod_{i=k+1}^{n} (1 - \alpha_i \beta_{i-1}) \right\} (1 + \alpha_k) x_k,
\end{equation}
where $\prod_{i={n+1}}^n$ evaluates to $1$.
Together with \eqref{eqn:recursive_summation_6op_perturbations_s} and
\eqref{eqn:recursive_summation_6op_perturbations_e}, \eqref{eqn:recursive_summation_6op_full_characterization} gives a complete characterization of the error behavior for Algorithm \ref{alg:recursive_summation_6op}.

Under the moderate assumption that $n\epsilon^2 < 1$ and using Lemma 3.1 of \cite{higham2002accuracy}, we can write
\begin{equation}
\label{eqn:recursive_summation_6op_full_characterization_using_theta}
\textstyle
s_n + e_n = \sum_{k=1}^n (1 + \theta_{n-k}) (1 + \alpha_k) x_k ,
\end{equation}
where $\vert \theta_m \vert \leq \frac{m\epsilon^2}{1 - m\epsilon^2}$ for $m = 0, \ldots, n$, which leads to the error bound below
\begin{subequations}
\label{eqn:recursive_summation_6op_sn+en_error_bound}
\begin{align}
\textstyle 
\left\vert s_n + e_n - \tilde{s} \right\vert 
& \textstyle 
= \left\vert \sum_{k=1}^n \alpha_k x_k + \theta_{n-k} x_k + \theta_{n-k} \alpha_k x_k \right\vert 
\label{eqn:recursive_summation_6op_sn+en_error_bound_a} \\
& \textstyle 
\leq \epsilon \sum_{k=1}^n \vert x_k \vert 
+ \sum_{k=1}^n \frac{(n-k)\epsilon^2}{1-(n-k)\epsilon^2} \vert x_k \vert 
+ \sum_{k=1}^n \frac{(n-k)\epsilon^3}{1-(n-k)\epsilon^2} \vert x_k \vert 
\label{eqn:recursive_summation_6op_sn+en_error_bound_b} \\
& \textstyle
\leq \epsilon \sum_{k=1}^n \vert x_k \vert 
+ \frac{(n-1)\epsilon^2}{1-(n-1)\epsilon^2} \sum_{k=1}^n \vert x_k \vert 
+ \frac{(n-1)\epsilon^3}{1-(n-1)\epsilon^2} \sum_{k=1}^n \vert x_k \vert \, . 
\label{eqn:recursive_summation_6op_sn+en_error_bound_c}
\end{align}
\end{subequations}
We note here that \eqref{eqn:recursive_summation_6op_sn+en_error_bound_c} adheres well to the principle of putting ``error bounds in a concise, easily interpreted form'' and is perhaps as informative as \eqref{eqn:recursive_summation_6op_sn+en_error_bound_b} for most application scenarios. 
One possible exception is when using compensated summation to verify hardware or software correctness, where a more precise bound such as \eqref{eqn:recursive_summation_6op_sn+en_error_bound_b} may bring additional value.

For some use cases, if we have the luxury to keep both $s_n$ and $e_n$ as the final result, then \eqref{eqn:recursive_summation_6op_sn+en_error_bound} already provides a useful error bound. 
If only $s_n$ is to be kept, we have from $s_n = (1 + \beta_n)(s_n + e_n)$ that
\begin{subequations}
\label{eqn:recursive_summation_6op_sn_error_bound}
\begin{align}
\vert s_n - \tilde{s} \vert
& \textstyle = \left\vert (1 + \beta_n)(s_n + e_n - \tilde{s}) + \beta_n\tilde{s} \right\vert 
\label{eqn:recursive_summation_6op_sn_error_bound_a} \\
& \textstyle \leq (1 + \epsilon) \left\{ \epsilon \sum_{k=1}^n \vert x_k \vert 
+ \frac{(n-1)\epsilon^2}{1-(n-1)\epsilon^2} \sum_{k=1}^n \vert x_k \vert 
+ \frac{(n-1)\epsilon^3}{1-(n-1)\epsilon^2} \sum_{k=1}^n \vert x_k \vert \right\} 
+ \epsilon \left\vert \sum_{k=1}^n x_k \right\vert \,.
\label{eqn:recursive_summation_6op_sn_error_bound_b}
\end{align}
\end{subequations}
We note here that the second term in the error bound above has the summation of addends inside the absolute value. 
The leading order term in the error bound above remains the same as in \eqref{eqn:recursive_summation_6op_sn+en_error_bound_c} for $\left\vert s_n + e_n - \tilde{s} \right\vert$, i.e., either $\mathcal{O}(\epsilon)$ or $\mathcal{O}(n \epsilon^2)$ depending on the size of $n \epsilon$.

\begin{remark}[$n$ and $\epsilon$]
\label{rmk:significance_of_n_epsilon^2}
We note here that the results \eqref{eqn:recursive_summation_6op_sn+en_error_bound} and \eqref{eqn:recursive_summation_6op_sn_error_bound} are derived under the assumption that $n\epsilon^2 < 1$.
For modern {\it large scale} applications potentially executed at {\it low precision}, where both $n$ and $\epsilon$ can be quite large, this is a significant improvement from the standard assumption $n\epsilon < 1$
in terms of the applicability of the results.
Indeed, the concern on the magnitude of $n$ and $\epsilon$ has been recognized and motivated the work on probabilistic error bounds; see, e.g., \cite{higham2019new}.

\hfill {\bfseries{\itshape End of Remark} \ref*{rmk:significance_of_n_epsilon^2}.}
\end{remark}

\begin{remark}[3op vs 6op]
\label{rmk:cost_on_modern_hardware}
On modern hardware, floating point operation count often has little impact on performance; it is the data movement that tends to dominate both the execution time and energy consumption. 
Time spent on small amounts of floating point operations can often be hidden entirely by data movement due to overlapping executions.
The additional floating point operation count in Algorithm \ref{alg:recursive_summation_6op} is therefore not a serious disadvantage compared to Algorithm \ref{alg:recursive_summation_3op}, given that they both involve the same additional variable $e$, whereas $y$ and those appearing inside \texttt{add\_3op} and \texttt{add\_6op} can be treated as temporary variables from the perspective of program translation and execution.

\hfill {\bfseries{\itshape End of Remark} \ref*{rmk:cost_on_modern_hardware}.}
\end{remark}

\subsection{Recursive summation with {\it double} 6op compensation} 
\label{sec:sumSequenceDouble6op}
Algorithm \ref{alg:recursive_summation_double_6op} improves on Algorithm \ref{alg:recursive_summation_6op} based on the observation that the {\it compensation} operation therein (i.e., line 4) 
is not compensated itself and carries an $\epsilon$ error, i.e., $\alpha_i$ in \eqref{eqn:recursive_summation_6op_perturbations_y} and \eqref{eqn:recursive_summation_6op_recursive_relation}. 
In Algorithm \ref{alg:recursive_summation_double_6op}, 
the operation $w = e + v$ (i.e., line 5) is still not compensated; but since both addends in this operation are errors, the error introduced when adding them is expected to be second order in nature.

\begin{center}
\begin{minipage}{0.6\linewidth}
\centering
\captionsetup[lstlisting]{font=small}
\begin{lstlisting}[
tabsize=4,
basicstyle=\ttfamily\small,
xleftmargin=6em,
caption={recursive summation with double 6op compensation.},
label={alg:recursive_summation_double_6op},
]
s = 0;
e = 0;
for i = 1:n
    [t,v] = add_6op(s,x[i]);
       w  = e + v;
    [s,e] = add_6op(t,w);
end
\end{lstlisting}
\end{minipage}
\end{center}

The analysis of Algorithm \ref{alg:recursive_summation_double_6op} follows the same strategy as that for Algorithm \ref{alg:recursive_summation_6op}. We start by deriving a recursive relation.
Attaching subscript index $i$ to disambiguate variables from different iterations,
we can write
\begin{subequations}
\label{eqn:recursive_summation_double_6op_perturbations}
\begin{align}
t_i & = (1+\alpha_i) (s_{i-1} + x_i) \label{eqn:recursive_summation_double_6op_perturbations_t} \\
v_i & =   -\alpha_i  (s_{i-1} + x_i) \label{eqn:recursive_summation_double_6op_perturbations_v} \\
w_i & = (1+\beta_i ) (e_{i-1} + v_i) \label{eqn:recursive_summation_double_6op_perturbations_w} \\
s_i & = (1+\gamma_i) (t_i     + w_i) \label{eqn:recursive_summation_double_6op_perturbations_s} \\
e_i & =  - \gamma_i  (t_i     + w_i) \label{eqn:recursive_summation_double_6op_perturbations_e} 
\end{align}
\end{subequations}
with $\vert \alpha_i \vert, \vert \beta_i \vert, \vert \gamma_i \vert \leq \epsilon$ following Lemma \ref{lemma:bound_on_rounding_with_rounding_inputs} and
\begin{subequations}
\label{eqn:recursive_summation_double_6op_recursive_relation}
\begin{align}
s_i + e_i 
& = t_i + w_i \\
& = t_i + e_{i-1} + v_i + \beta_i(e_{i-1} + v_i) \\
& = s_{i-1} + x_i + e_{i-1} + \beta_i e_{i-1} - \alpha_i \beta_i (s_{i-1} + x_i) 
\label{eqn:recursive_summation_double_6op_recursive_relation_c} \\
& = (s_{i-1} + e_{i-1}) + (1 - \alpha_i \beta_i) x_i - \beta_i \gamma_{i-1} (s_{i-1} + e_{i-1}) - \alpha_i \beta_i(1 + \gamma_{i-1}) (s_{i-1} + e_{i-1}) 
\label{eqn:recursive_summation_double_6op_recursive_relation_d} \\
& = (1 - \alpha_i\beta_i - \beta_i \gamma_{i-1} - \alpha_i\beta_i\gamma_{i-1}) (s_{i-1} + e_{i-1}) + (1 - \alpha_i \beta_i) x_i 
\label{eqn:recursive_summation_double_6op_recursive_relation_e} \,.
\end{align}
\end{subequations} 
From \eqref{eqn:recursive_summation_double_6op_recursive_relation_c} to \eqref{eqn:recursive_summation_double_6op_recursive_relation_d}, we used relations $e_{i-1} = - \gamma_{i-1}(s_{i-1} + e_{i-1})$ and $s_{i-1} = (1 + \gamma_{i-1})(s_{i-1} + e_{i-1})$.
Notice that the perturbation in the $x_i$ term above is of order $\epsilon^2$ instead of $\epsilon$ in \eqref{eqn:recursive_summation_6op_recursive_relation} for Algorithm \ref{alg:recursive_summation_6op}.

Applying \eqref{eqn:recursive_summation_double_6op_recursive_relation_e} recursively, we arrive at
\begin{equation}
\label{eqn:recursive_summation_double_6op_full_characterization}
\textstyle
s_n + e_n = \sum_{k=1}^n \left\{ \prod_{i=k+1}^n ( 1 - \alpha_i \beta_i - \beta_i \gamma_{i-1} - \alpha_i \beta_i \gamma_{i-1} ) \right\} (1 - \alpha_k \beta_k) x_k \,,
\end{equation}
where $\prod_{i=n+1}^n$ evaluates to $1$. Together with \eqref{eqn:recursive_summation_double_6op_perturbations_s} and \eqref{eqn:recursive_summation_double_6op_perturbations_e}, \eqref{eqn:recursive_summation_double_6op_full_characterization} gives a complete characterization of the error behavior for Algorithm \ref{alg:recursive_summation_double_6op}.

Introducing $\sigma_i = - \alpha_i \beta_i - \beta_i \gamma_{i-1} - \alpha_i \beta_i \gamma_{i-1}$ and $\tau_k = - \alpha_k \beta_k$ to simplify the notation, we have 
$\vert \sigma_i \vert \leq \sigma \coloneqq 2 \epsilon^2 + \epsilon^3$ for $i = 2, \ldots, n$ and $\vert \tau_k \vert \leq \tau \coloneqq \epsilon^2$ for $k = 1, \ldots, n$.
Under the assumption that $n \sigma < 1$ and using Lemma 3.1 of \cite{higham2002accuracy}, we can write
\begin{equation}
\label{eqn:recursive_summation_double_6op_full_characterization_using_theta}
\textstyle
s_n + e_n = \sum_{k=1}^n (1 + \theta_{n-k}) (1 + \tau_k) x_k \,,
\end{equation}
where $\vert \theta_m \vert \leq \frac{m\sigma}{1 - m\sigma}$ for $m = 0, \ldots, n$,
which leads to the error bound below
\begin{subequations}
\label{eqn:recursive_summation_double_6op_sn+en_error_bound}
\begin{align}
\vert s_n + e_n - \tilde{s} \vert
& \textstyle
= \left\vert \sum_{k=1}^n ( \tau_k x_k + \theta_{n-k} x_k + \theta_{n-k} \tau_k x_k ) \right\vert 
\label{eqn:recursive_summation_double_6op_sn+en_error_bound_a} \\
& \textstyle
\leq \tau \sum_{k=1}^n \vert x_k \vert 
+ \sum_{k=1}^n \frac{(n-k)\sigma}{1-(n-k)\sigma} \vert x_k \vert 
+ \sum_{k=1}^n \frac{(n-k)\sigma\tau}{1-(n-k)\sigma} \vert x_k \vert 
\label{eqn:recursive_summation_double_6op_sn+en_error_bound_b} \\
& \textstyle
\leq \tau \sum_{k=1}^n \vert x_k \vert 
+ \frac{(n-1)\sigma}{1-(n-1)\sigma} \sum_{k=1}^n \vert x_k \vert 
+ \frac{(n-1)\sigma\tau}{1-(n-1)\sigma} \sum_{k=1}^n \vert x_k \vert \,. 
\label{eqn:recursive_summation_double_6op_sn+en_error_bound_c}
\end{align}
\end{subequations}
Comparing \eqref{eqn:recursive_summation_double_6op_sn+en_error_bound_c} to \eqref{eqn:recursive_summation_6op_sn+en_error_bound_c}, 
the improvement in the bound comes from the first term, i.e., from $\epsilon \sum_{k=1}^n \vert x_k \vert$ to $\tau \sum_{k=1}^n \vert x_k \vert$, where $\tau = \epsilon^2$, 
which will lead to significant improvement in the overall bound when the first term in \eqref{eqn:recursive_summation_6op_sn+en_error_bound_c} is dominating, i.e., when $n \epsilon \ll 1$.
If focusing on the leading order term only, we have
\begin{equation}
\label{eqn:recursive_summation_double_6op_sn+en_error_bound_leading_order}
\textstyle
\vert s_n + e_n - \tilde{s} \vert \lesssim (2n - 1) \epsilon^2 \sum_{k=1}^n \vert x_k \vert \,.  
\end{equation}
If only $s_n$ is to be kept, we have from $s_n = (1 + \gamma_n)(s_n + e_n)$ that 
\begin{subequations}
\label{eqn:recursive_summation_double_6op_sn_error_bound}
\begin{align}
\vert s_n - \tilde{s} \vert
& \textstyle = \left\vert (1 + \gamma_n)(s_n + e_n - \tilde{s}) + \gamma_n\tilde{s} \right\vert 
\label{eqn:recursive_summation_double_6op_sn_error_bound_a} \\
& \textstyle \leq (1 + \epsilon) \left\{ \tau \sum_{k=1}^n \vert x_k \vert 
+ \frac{(n-1)\sigma}{1-(n-1)\sigma} \sum_{k=1}^n \vert x_k \vert 
+ \frac{(n-1)\sigma\tau}{1-(n-1)\sigma} \sum_{k=1}^n \vert x_k \vert \right\} 
+ \epsilon \left\vert \sum_{k=1}^n x_k \right\vert \,.
\label{eqn:recursive_summation_double_6op_sn_error_bound_b}
\end{align}
\end{subequations}
When $n \epsilon \ll 1$ or, putting it more precisely, 
when $n \epsilon^2 \sum_{k=1}^{n} \vert x_k \vert \lesssim \epsilon \left\vert \sum_{k=1}^{n} x_k \right\vert$, 
the leading order term in the error bound has improved from $\epsilon \sum_{k=1}^n \vert x_k \vert$ in \eqref{eqn:recursive_summation_6op_sn_error_bound_b} to $\epsilon \left\vert \sum_{k=1}^n x_k \right\vert$ in \eqref{eqn:recursive_summation_double_6op_sn_error_bound_b}.
In this case, 
\eqref{eqn:recursive_summation_double_6op_sn_error_bound_b} offers a relative error bound of $\epsilon$ on $\vert s_n - \tilde{s} \vert$, which indicates that Algorithm \ref{alg:recursive_summation_double_6op} can be viewed as carrying out the summation sequence with one rounding only, regardless of the size of $n$.

\subsection{Recursive summation with {\it triple} 6op compensation} 
\label{sec:sumSequenceTriple6op}
Algorithm \ref{alg:recursive_summation_triple_6op} is directly inspired by the work of Priest \cite{priest1992properties}, 
which uses three 3op compensations per addend and requires the addends to be ordered by decreasing magnitude.
As revealed by the analysis below, this ordering can be removed by replacing the 3op compensations with the 6op versions. 
The removal of this ordering requirement is important for use cases where the new addends are generated on the fly, possibly depending on the current summing result, e.g., the case of solution update in time stepping schemes or iterative solvers.

\begin{center}
\begin{minipage}{0.6\linewidth}
\centering
\captionsetup[lstlisting]{font=small}
\begin{lstlisting}[
tabsize=4,
basicstyle=\ttfamily\small,
xleftmargin=6em,
caption={recursive summation with triple 6op compensation.},
label={alg:recursive_summation_triple_6op},
]
s = 0;
e = 0;
for i = 1:n
    [y,u] = add_6op(e,x[i]);
    [t,v] = add_6op(s,y);
       w  = u + v;
    [s,e] = add_6op(t,w);
end
\end{lstlisting}
\end{minipage}
\end{center}

The analysis of Algorithm \ref{alg:recursive_summation_triple_6op} follows the same strategy as those for Algorithms \ref{alg:recursive_summation_6op} and \ref{alg:recursive_summation_double_6op}. We start by deriving a recursive relation.
Attaching subscript index $i$ to disambiguate variables from different iterations,
we can write
\begin{subequations}
\label{eqn:recursive_summation_triple_6op_perturbations}
\begin{align}
y_i & = (1+\alpha_i) (e_{i-1} + x_i) \label{eqn:recursive_summation_triple_6op_perturbations_y} \\
u_i & =   -\alpha_i  (e_{i-1} + x_i) \label{eqn:recursive_summation_triple_6op_perturbations_u} \\
t_i & = (1+\beta_i ) (s_{i-1} + y_i) \label{eqn:recursive_summation_triple_6op_perturbations_t} \\
v_i & =   -\beta_i   (s_{i-1} + y_i) \label{eqn:recursive_summation_triple_6op_perturbations_v} \\
w_i & = (1+\gamma_i) (u_i     + v_i) \label{eqn:recursive_summation_triple_6op_perturbations_w} \\
s_i & = (1+\eta_i  ) (t_{i}   + w_i) \label{eqn:recursive_summation_triple_6op_perturbations_s} \\
e_i & =  - \eta_i    (t_{i}   + w_i) \label{eqn:recursive_summation_triple_6op_perturbations_e} 
\end{align}
\end{subequations}
with $\vert \alpha_i \vert, \vert \beta_i \vert, \vert \gamma_i \vert, \vert \eta_i \vert \leq \epsilon$ following Lemma \ref{lemma:bound_on_rounding_with_rounding_inputs} and
\begin{subequations}
\label{eqn:recursive_summation_triple_6op_recursive_relation}
\begin{align}
s_i + e_i 
& = t_i + w_i \\
& = t_i + u_i + v_i + \gamma_i(u_i + v_i) \\
& = s_{i-1} + y_i + u_i + \gamma_i(u_i + v_i) \\
& = s_{i-1} + e_{i-1} + x_i + \gamma_i(u_i + v_i) \\
& = (s_{i-1} + e_{i-1}) + x_i - \gamma_i[\alpha_i(e_{i-1} + x_i) + \beta_i(s_{i-1}+y_i)] \\
& = (s_{i-1} + e_{i-1}) + x_i - \gamma_i \alpha_i e_{i-1} - \gamma_i \alpha_i x_i 
  - \gamma_i \beta_i( s_{i-1} + e_{i-1} + \alpha_i e_{i-1} + x_i + \alpha_i x_i ) \\
& = (1 - \gamma_i\beta_i)(s_{i-1} + e_{i-1}) - (\gamma_i\alpha_i + \gamma_i\beta_i\alpha_i) e_{i-1} + (1 - \gamma_i\alpha_i - \gamma_i\beta_i - \gamma_i\beta_i\alpha_i) x_i
\label{eqn:recursive_summation_triple_6op_recursive_relation_g} \\
& = \left(1 - \gamma_i\beta_i + \gamma_i\alpha_i\eta_{i-1} + \gamma_i\beta_i\alpha_i\eta_{i-1}\right)(s_{i-1} + e_{i-1}) + (1 - \gamma_i\alpha_i - \gamma_i\beta_i - \gamma_i\beta_i\alpha_i)x_i \, .
%
\label{eqn:recursive_summation_triple_6op_recursive_relation_h}
\end{align}
\end{subequations}
From \eqref{eqn:recursive_summation_triple_6op_recursive_relation_g} to \eqref{eqn:recursive_summation_triple_6op_recursive_relation_h}, we used the relation $e_{i-1} = - \eta_{i-1} (s_{i-1} + e_{i-1})$.
Applying \eqref{eqn:recursive_summation_triple_6op_recursive_relation_h} recursively, we arrive at
\begin{equation}
\label{eqn:recursive_summation_triple_6op_full_characterization}
\textstyle
s_n + e_n = \sum_{k=1}^n\left\{ \prod_{i=k+1}^n \left( 1 - \gamma_i\beta_i + \gamma_i\alpha_i\eta_{i-1} + \gamma_i\beta_i\alpha_i\eta_{i-1} \right) \right\} 
\left(1 - \gamma_k\alpha_k - \gamma_k\beta_k - \gamma_k\beta_k\alpha_k\right) x_k,
\end{equation}
where $\prod_{i=n+1}^n$ evaluates to $1$. Together with \eqref{eqn:recursive_summation_triple_6op_perturbations_s} and \eqref{eqn:recursive_summation_triple_6op_perturbations_e}, \eqref{eqn:recursive_summation_triple_6op_full_characterization} gives a complete characterization of the error behavior for Algorithm \ref{alg:recursive_summation_triple_6op}.

Introducing $\sigma_i = - \gamma_i\beta_i + \gamma_i\alpha_i\eta_{i-1} + \gamma_i\beta_i\alpha_i\eta_{i-1}$ 
and $\tau_k = - \gamma_k\alpha_k - \gamma_k\beta_k - \gamma_k\beta_k\alpha_k$ to simplify the notation, 
we have $\vert \sigma_i \vert \leq \sigma \coloneqq \epsilon^2 + \epsilon^3 + \epsilon^4$ for $i=2, \ldots, n$ 
and $\vert \tau_k \vert \leq \tau \coloneqq 2 \epsilon^2 + \epsilon^3$ for $k=1,\ldots,n$.
Under the assumption that $n\sigma < 1$ and using Lemma 3.1 of \cite{higham2002accuracy}, we can write
\begin{equation}
\label{eqn:recursive_summation_triple_6op_full_characterization_using_theta}
\textstyle
s_n + e_n = \sum_{k=1}^n (1 + \theta_{n-k}) (1 + \tau_k) x_k \,,
\end{equation}
where $\vert \theta_m \vert \leq \frac{m\sigma}{1 - m\sigma}$ for $m = 0, \ldots, n$,
which leads to the error bound below
\begin{equation}
\label{eqn:recursive_summation_triple_6op_sn+en_error_bound}
\vert s_n + e_n - \tilde{s} \vert
\textstyle
\leq \tau \sum_{k=1}^n \vert x_k \vert 
+ \frac{(n-1)\sigma}{1-(n-1)\sigma} \sum_{k=1}^n \vert x_k \vert 
+ \frac{(n-1)\sigma\tau}{1-(n-1)\sigma} \sum_{k=1}^n \vert x_k \vert \,. 
\end{equation}
We note here that the bound given in \eqref{eqn:recursive_summation_triple_6op_sn+en_error_bound} has the same formula as \eqref{eqn:recursive_summation_double_6op_sn+en_error_bound_c}, but with $\sigma$ and $\tau$ taking new meanings.
If focusing on the leading order term only, we have 
\begin{equation}
\label{eqn:recursive_summation_triple_6op_sn+en_error_bound_leading_order}
\textstyle
\vert s_n + e_n - \tilde{s} \vert \lesssim (n+1) \epsilon^2 \sum_{k=1}^n \vert x_k \vert \,,
\end{equation}
which has the same order as the bound from \eqref{eqn:recursive_summation_double_6op_sn+en_error_bound_leading_order}, but with an improved constant.
When 
compared to \eqref{forwardErrorBoundRecursiveSummation}, this bound also indicates that Algorithm \ref{alg:recursive_summation_triple_6op} may be viewed as carrying out the summation as if we have used Algorithm \ref{alg:recursive_summation} but with a datatype that has twice the mantissa bits. 
Finally, if only $s_n$ is to be kept, the same error bound as expressed in \eqref{eqn:recursive_summation_double_6op_sn_error_bound_b} can be derived for $\vert s_n - \tilde{s} \vert$, with $\sigma$ and $\tau$ taking new meanings, and the same observation that follows \eqref{eqn:recursive_summation_double_6op_sn_error_bound_b} can be made when the bound is compared to \eqref{eqn:recursive_summation_6op_sn_error_bound_b}.

\section{Numerical examples}
\label{sec:numericalExamples}

\subsection{Pedagogical}
We first present a couple of 
pedagogical examples to illustrate the inner workings of the various compensated summation algorithms presented in Section \ref{sec:compensatedSum}.
The building block of these constructed examples is the operations $\beta^{t+k-1} \pm \beta^0$, which, $\forall k>1$, 
give the outcome $\beta^{t+k-1}$ in floating point arithmetic with the second addend dropped. 
It can be easily verified that with $x=\beta^{t+k-1}$ and $y=\pm \beta^0$, the procedure from Theorem \ref{thm:3opEFTadditionDekker} gives $z=x$ and $zz=y$; whereas with $y=\beta^{t+k-1}$ and $x=\pm \beta^0$, it gives $z=y$ and $zz=0$ instead.  
On the other hand, the procedure from Theorem \ref{thm:6opEFTadditionDekker} gives $z=\beta^{t+k-1}$ and $zz=\pm \beta^0$ in both scenarios.

The sequence $x_1 = 2^{t+1}, x_2 = -2^0, x_3 = -2^0$ can be used to illustrate the difference between the non-compensated summing procedure in Algorithm \ref{alg:recursive_summation} and the two compensated ones in Algorithms \ref{alg:recursive_summation_3op} and \ref{alg:recursive_summation_6op}.
It is clear that with Algorithm \ref{alg:recursive_summation}, the output $s$ remains $x_1$ because the subsequent addends have too smalls magnitudes to be added on.
On the other hand, Algorithms \ref{alg:recursive_summation_3op} and \ref{alg:recursive_summation_6op} produce the following results.
\begin{center}
\begin{tabular}{rp{3.75em}p{7em}p{3.75em}} 
$i=1:$ & $y=2^{t+1}$ & $s=2^{t+1}           $ & $e=0$ \\
$i=2:$ & $y= -2^{0}$ & $s=2^{t+1}           $ & $e=-2^{0}$ \\
$i=3:$ & $y= -2^{1}$ & $s=(2^t-1) \times 2^1$ & $e = 0$ \\
\end{tabular}
\end{center}
In this case, the error from step $i=2$ and the addend from step $i=3$ combined into a number ($y= -2^{1}$) that has a large enough magnitude to be added to $s$.
In fact, the calculations of $y$ and $s$ from step $i=3$ are both exact, leading to the exact answer in the final output $s$.

The difference between the effects of Algorithms \ref{alg:recursive_summation_3op} and \ref{alg:recursive_summation_6op} is easy to comprehend by comparing the conditions required for Theorems \ref{thm:3opEFTadditionDekker} and \ref{thm:6opEFTadditionDekker}; hence, its illustration is omitted.
Instead, we use the sequence $x_1 = 2^0, x_2 = 2^{t+1}, x_3 = -2^{t+1}, x_4 = -2^0$ to illustrate the difference between Algorithm \ref{alg:recursive_summation_6op}, which compensates only the operation that accumulates into the running sum $s$, and the two Algorithms \ref{alg:recursive_summation_double_6op} and \ref{alg:recursive_summation_triple_6op} that apply additional compensations.
For this sequence, Algorithm \ref{alg:recursive_summation_6op} produces the following results.
\begin{center}
\begin{tabular}{rp{4.5em}p{4.5em}p{4.5em}} 
$i=1:$ & $y= 2^0$     & $s=2^0    $ & $e=0$   \\
$i=2:$ & $y= 2^{t+1}$ & $s=2^{t+1}$ & $e=2^0$ \\
$i=3:$ & $y=-2^{t+1}$ & $s=0      $ & $e=0$   \\
$i=4:$ & $y=-2^0$     & $s=-2^0   $ & $e=0$   
\end{tabular}
\end{center}
The exact answer should obviously be zero, but we obtain $-2^0$ in the final outcome using Algorithm \ref{alg:recursive_summation_6op}. 
The issue occurred at step $i=3$ when calculating $y$: the leftover error from the previous step ($2^0$) is too small, in the relative sense, to be added to the current input ($-2^{t+1}$) and consequently lost in the sum.

In comparison, Algorithm \ref{alg:recursive_summation_double_6op} produces the following results.
\begin{center}
\begin{tabular}{rp{3.5em}p{3.5em}p{3.5em}p{3.5em}p{3.5em}} 
$i=1:$ & $t=2^0$     & $v=0$   & $w=0$   & $s=2^0$     & $e=0$   \\
$i=2:$ & $t=2^{t+1}$ & $v=2^0$ & $w=2^0$ & $s=2^{t+1}$ & $e=2^0$ \\
$i=3:$ & $t=0$       & $v=0$   & $w=2^0$ & $s=2^0$     & $e=0$   \\
$i=4:$ & $t=0$       & $v=0$   & $w=0$   & $s=0$       & $e=0$   
\end{tabular}
\end{center}
The relatively small error from step $i=2$ is kept in $w$ and added to the running sum $s$ successfully, leading to the correct answer in the final output.
Similarly, Algorithm \ref{alg:recursive_summation_triple_6op} produces the following results.
\begin{center}
\begin{tabular}{rp{4em}p{3em}p{3.5em}p{3em}p{3em}p{3.5em}p{3em}} 
$i=1:$ & $y=2^0$ & $u=0$ & $t=2^0$ & $v=0$ & $w=0$ & $s=2^0$ & $e=0$ \\
$i=2:$ & $y=2^{t+1}$ & $u=0$ & $t=2^{t+1}$ & $v=2^0$ & $w=2^0$ & $s=2^{t+1}$ & $e=2^0$ \\
$i=3:$ & $y=-2^{t+1}$ & $u=2^0$ & $t=0$ & $v=0$ & $w=2^0$ & $s=2^0$ & $e=0$ \\
$i=4:$ & $y=-2^0$ & $u=0$ & $t=0$ & $v=0$ & $w=0$ & $s=0$ & $e=0$   
\end{tabular}
\end{center}
The relatively small error from step $i=2$ is kept in $u$ and then $w$ and added to the running sum $s$ successfully, leading to the correct answer in the final output.

\subsection{Accumulation}
\label{example:accumulation}
An accumulation operation of the form $s = \sum_{i=1}^{n} x_i$ is often needed in various computational workloads. Common examples include Monte Carlo simulations and the collection of statistics.
%
Through private communications, accuracy of this accumulation operation has been raised as a major concern for Monte Carlo applications such as QMCPACK \cite{kim2018qmcpack} and OpenMC \cite{romano2015openmc} on hardware that do not offer 
enough 
native double precision support.
In this example, we draw samples of random numbers $\{x_i\}_{i=1}^n$ and experiment with the various algorithms from Section \ref{sec:compensatedSum} to calculate their sum.
For the experiments in this example, the random numbers are drawn as uniformly distributed bit sequences 
with certain exponent patterns filtered out, including the all-one exponent that corresponds to \texttt{Inf} and \texttt{NaN} and a few additional largest exponents to avoid overflow.

Relative errors corresponding to Algorithm \ref{alg:recursive_summation} (no compensation), Algorithm \ref{alg:recursive_summation_6op} (6op compensation), and Algorithm \ref{alg:recursive_summation_double_6op} (double 6op compensation) are summarized in Table \ref{tbl:accumulation_fp32_relative_error} and Table \ref{tbl:accumulation_fp64_relative_error} for single and double precision calculations, respectively. 
It can be readily observed that the two compensated algorithms deliver better accuracy. In particular, when compared to Algorithm \ref{alg:recursive_summation} (no compensation), Algorithm \ref{alg:recursive_summation_double_6op} (double 6op compensation) delivers the accuracy as if a data format with twice the amount of mantissa bits was used.

\newcommand{\smalltt}[1]{{\small\texttt{#1}}}

\begin{table}[H]
\centering
\caption{Relative errors for calculating $s = \sum_{i=1}^{n} x_i$ using Algorithm \ref{alg:recursive_summation}, Algorithm \ref{alg:recursive_summation_6op}, and Algorithm \ref{alg:recursive_summation_double_6op} at \textit{single} precision. Bit sequences of 32 bits are drawn uniformly; sequences with exponents equal to or above \texttt{1111 0111} are discarded to avoid overflow.}
\label{tbl:accumulation_fp32_relative_error}
\begin{tabular}{cccc} 
\hline
& no compensation & 6op & double 6op \\
\hline\\[-1em]
$2^{2} $ & \texttt{2.7232E-08} & \texttt{5.8741E-19} & \texttt{5.8741E-19} \\
$2^{4} $ & \texttt{8.0680E-09} & \texttt{7.5526E-15} & \texttt{1.3645E-16} \\
$2^{6} $ & \texttt{3.6740E-08} & \texttt{1.2816E-08} & \texttt{3.3204E-16} \\
$2^{8} $ & \texttt{1.0608E-07} & \texttt{3.4012E-09} & \texttt{4.4701E-15} \\
$2^{10}$ & \texttt{5.0575E-07} & \texttt{5.0312E-08} & \texttt{2.4471E-14} \\
$2^{12}$ & \texttt{1.6722E-07} & \texttt{9.2296E-09} & \texttt{1.3122E-14} \\
$2^{14}$ & \texttt{5.4611E-07} & \texttt{3.0616E-10} & \texttt{4.0441E-14} \\
$2^{16}$ & \texttt{4.9748E-06} & \texttt{2.2910E-08} & \texttt{1.8646E-14} \\
$2^{18}$ & \texttt{1.7089E-06} & \texttt{1.9204E-09} & \texttt{1.3608E-14} \\
$2^{20}$ & \texttt{5.7614E-06} & \texttt{1.3258E-08} & \texttt{4.2820E-13} \\
\hline
\end{tabular}
\end{table}

\vspace{-1em}
\begin{table}[H]
\centering
\caption{Relative errors for calculating $s = \sum_{i=1}^{n} x_i$ using Algorithm \ref{alg:recursive_summation}, Algorithm \ref{alg:recursive_summation_6op}, and Algorithm \ref{alg:recursive_summation_double_6op} at \textit{double} precision. Bit sequences of 64 bits are drawn uniformly; sequences with exponents equal to or above \texttt{1111 1010} are discarded to avoid overflow.}
\label{tbl:accumulation_fp64_relative_error}
\begin{tabular}{cccc} 
\hline
& no compensation & 6op & double 6op \\
\hline\\[-1em]
$2^{2} $ & \texttt{3.0810E-17} & \texttt{0.0000E+00} & \texttt{0.0000E+00} \\
$2^{4} $ & \texttt{4.7274E-53} & \texttt{1.3203E-69} & \texttt{1.3203E-69} \\
$2^{6} $ & \texttt{4.4697E-17} & \texttt{5.7065E-20} & \texttt{9.3778E-34} \\
$2^{8} $ & \texttt{6.9736E-17} & \texttt{3.9433E-18} & \texttt{1.1858E-32} \\
$2^{10}$ & \texttt{1.6921E-16} & \texttt{2.1145E-18} & \texttt{3.7625E-33} \\
$2^{12}$ & \texttt{1.6312E-15} & \texttt{2.2531E-17} & \texttt{4.7920E-32} \\
$2^{14}$ & \texttt{2.2156E-16} & \texttt{1.5226E-16} & \texttt{2.9948E-33} \\
$2^{16}$ & \texttt{8.2921E-16} & \texttt{4.3061E-18} & \texttt{2.5255E-32} \\
$2^{18}$ & \texttt{1.4821E-15} & \texttt{1.6615E-17} & \texttt{1.2419E-31} \\
$2^{20}$ & \texttt{8.4132E-15} & \texttt{2.7501E-17} & \texttt{1.3656E-30} \\
\hline
\end{tabular}
\end{table}

\begin{remark}
\label{rmk:reference_sum_in_place_for_true_sum}
The reference sum used in the above error calculation is obtained by carrying out the summation using Algorithm \ref{alg:recursive_summation_double_6op} (double 6op compensation) with the extended 80-bit data format available on x86 architecture.
This 80-bit data format allocates 64 bits for the mantissa, hence having unit round-off $\epsilon_{80} = 2^{-64}$. In comparison, the single and double precision data formats have unit round-off $\epsilon_{32} = 2^{-24}$ and $\epsilon_{64} = 2^{-53}$, respectively.
Although the reference sum, denoted hereafter as $\mathcal{S}_\text{ref}$\footnotemark, 
is different from the true sum $\tilde{s}$, we have from triangle inequality that
$\vert \mathcal{S} - \tilde{s} \vert 
\leq 
\vert \mathcal{S} - \mathcal{S}_\text{ref} \vert + \vert \mathcal{S}_\text{ref} - \tilde{s} \vert$.
Since $\vert \mathcal{S}_\text{ref} - \tilde{s} \vert$ is calculated using a very accurate algorithm with the extended 80-bit precision, we can expect it to amount to a small perturbation only in relation to $\vert \mathcal{S} - \tilde{s} \vert$ and $\vert \mathcal{S} - \mathcal{S}_\text{ref} \vert$.
To elaborate, $\vert \mathcal{S}_\text{ref} - \tilde{s} \vert$ is bounded by 
$\epsilon_{80}^2 \cdot (2n-1) \sum_{k=1}^n \vert x_k \vert$ according to 
\eqref{eqn:recursive_summation_double_6op_sn+en_error_bound_leading_order}.
Among all the experiments from Tables \ref{tbl:accumulation_fp32_relative_error} and \ref{tbl:accumulation_fp64_relative_error}, the most accurate sum calculation uses Algorithm \ref{alg:recursive_summation_double_6op} (double 6op compensation) with double precision, which leads to a similar bound $\epsilon_{64}^2 \cdot (2n-1) \sum_{k=1}^n \vert x_k \vert$ on $\vert \mathcal{S} - \tilde{s} \vert$, a factor of $2^{-22} \sim 10^{-7}$ greater than that of $\vert \mathcal{S}_\text{ref} - \tilde{s} \vert$. 
Therefore, it is reasonable to consider the reference sum as a good substitute for the true sum.
Indeed, numerical observation confirms that, with the exception of the double 6op column in Table 2 when $n$ is very small ($n \leq 2^4)$, even the very pessimistic bound on $\vert \mathcal{S}_\text{ref} - \tilde{s} \vert$ is small enough compared to $\vert \mathcal{S} - \mathcal{S}_\text{ref} \vert$ to not affect the exponents reported in Tables \ref{tbl:accumulation_fp32_relative_error} and \ref{tbl:accumulation_fp64_relative_error}.

\hfill {\bfseries{\itshape End of Remark} \ref*{rmk:reference_sum_in_place_for_true_sum}.}
\end{remark}

\footnotetext{
We use the symobl $\mathcal{S}$ here to indicate that the calculated sum is possibly represented by two numbers $s$ and $e$ as in the cases of Algorithms \ref{alg:recursive_summation_6op} and \ref{alg:recursive_summation_double_6op}.
}

The theoretical bounds derived for $\frac{\left\vert s + e - \mathcal{S}_{ref} \right\vert}{\sum_{i=1}^{n} \vert x_i\vert}$ are also of interest. 
Table \ref{tbl:accumulation_fp32_bounds} collects these bounds for Algorithms \ref{alg:recursive_summation}, \ref{alg:recursive_summation_6op}, and \ref{alg:recursive_summation_double_6op} at single precision based on \eqref{forwardErrorBoundRecursiveSummation}, \eqref{eqn:recursive_summation_6op_sn+en_error_bound_c}, and \eqref{eqn:recursive_summation_double_6op_sn+en_error_bound_c}, respectively, along with the actual observed values;
Table \ref{tbl:accumulation_fp64_bounds} does the same for these algorithms at double precision.
If the observed value exceeds the bound, it indicates that some unintended error has occurred during the calculation, perhaps related to hardware defects, software bugs, or data corruption.
Therefore, the algorithms presented in Section \ref{sec:compensatedSum}, along with the bounds given therein, can be leveraged to validate computer systems. 
%
Obviously, the smaller the bound, the better it is at detecting unintended errors.
%
To compare, the double 6op algorithm applied to adding $2^{20}$ random numbers at single and double precision has bounds on the scale of $10^{-9}$ and $10^{-26}$, respectively, much smaller than the $10^{-2}$ and $10^{-10}$ offered by the simple algorithm with no compensation, hence a more effective algorithm for the purpose of detecting unintended errors.

\vspace{-1em}
\begin{table}[H]
\centering
\caption{Bounds for $\frac{\left\vert s + e - \mathcal{S}_{ref} \right\vert}{\sum_{i=1}^{n} \vert x_i\vert}$ using Algorithm \ref{alg:recursive_summation}, Algorithm \ref{alg:recursive_summation_6op}, and Algorithm \ref{alg:recursive_summation_double_6op} at \textit{single} precision. 
}
\label{tbl:accumulation_fp32_bounds}
\begin{tabular}{ccccccc} 
\hline
    & \multicolumn{2}{c}{no compensation} & \multicolumn{2}{c}{6op} & \multicolumn{2}{c}{double 6op} \\
$n$ & derived & observed & derived & observed & derived & observed \\
\hline\\[-1em]
$2^{2} $ & \texttt{2.38E-07} & \texttt{2.72E-08} & \texttt{5.96E-08} & \texttt{5.87E-19} & \texttt{2.49E-14} & \texttt{5.87E-19} \\
$2^{4} $ & \texttt{9.54E-07} & \texttt{8.07E-09} & \texttt{5.96E-08} & \texttt{7.55E-15} & \texttt{1.10E-13} & \texttt{1.36E-16} \\
$2^{6} $ & \texttt{3.81E-06} & \texttt{3.67E-08} & \texttt{5.96E-08} & \texttt{1.28E-08} & \texttt{4.51E-13} & \texttt{3.32E-16} \\
$2^{8} $ & \texttt{1.53E-05} & \texttt{1.06E-07} & \texttt{5.96E-08} & \texttt{3.39E-09} & \texttt{1.82E-12} & \texttt{4.45E-15} \\
$2^{10}$ & \texttt{6.10E-05} & \texttt{5.79E-08} & \texttt{5.96E-08} & \texttt{5.76E-09} & \texttt{7.27E-12} & \texttt{2.80E-15} \\
$2^{12}$ & \texttt{2.44E-04} & \texttt{5.65E-08} & \texttt{5.96E-08} & \texttt{3.12E-09} & \texttt{2.91E-11} & \texttt{4.44E-15} \\
$2^{14}$ & \texttt{9.78E-04} & \texttt{8.29E-08} & \texttt{5.97E-08} & \texttt{4.65E-11} & \texttt{1.16E-10} & \texttt{6.14E-15} \\
$2^{16}$ & \texttt{3.92E-03} & \texttt{1.15E-07} & \texttt{5.98E-08} & \texttt{5.30E-10} & \texttt{4.66E-10} & \texttt{4.32E-16} \\
$2^{18}$ & \texttt{1.59E-02} & \texttt{2.81E-08} & \texttt{6.05E-08} & \texttt{3.16E-11} & \texttt{1.86E-09} & \texttt{2.24E-16} \\
$2^{20}$ & \texttt{6.67E-02} & \texttt{2.46E-08} & \texttt{6.33E-08} & \texttt{5.66E-11} & \texttt{7.45E-09} & \texttt{1.83E-15} \\
\hline
\end{tabular}
\end{table}

\vspace{-1.5em}
\begin{table}[H]
\centering
\caption{Bounds for $\frac{\left\vert s + e - \mathcal{S}_{ref} \right\vert}{\sum_{i=1}^{n} \vert x_i\vert}$ using Algorithm \ref{alg:recursive_summation}, Algorithm \ref{alg:recursive_summation_6op}, and Algorithm \ref{alg:recursive_summation_double_6op} at \textit{double} precision. 
}
\label{tbl:accumulation_fp64_bounds}
\begin{tabular}{ccccccc} 
\hline
    & \multicolumn{2}{c}{no compensation} & \multicolumn{2}{c}{6op} & \multicolumn{2}{c}{double 6op} \\
$n$ & derived & observed & derived & observed & derived & observed \\
\hline\\[-1em]
$2^{2} $ & \texttt{4.44E-16} & \texttt{3.08E-17} & \texttt{1.11E-16} & \texttt{0.00E+00} & \texttt{8.63E-32} & \texttt{0.00E+00} \\
$2^{4} $ & \texttt{1.78E-15} & \texttt{4.73E-53} & \texttt{1.11E-16} & \texttt{1.32E-69} & \texttt{3.82E-31} & \texttt{1.32E-69} \\
$2^{6} $ & \texttt{7.11E-15} & \texttt{4.46E-17} & \texttt{1.11E-16} & \texttt{5.69E-20} & \texttt{1.57E-30} & \texttt{9.36E-34} \\
$2^{8} $ & \texttt{2.84E-14} & \texttt{6.97E-17} & \texttt{1.11E-16} & \texttt{3.94E-18} & \texttt{6.30E-30} & \texttt{1.19E-32} \\
$2^{10}$ & \texttt{1.14E-13} & \texttt{1.57E-16} & \texttt{1.11E-16} & \texttt{1.96E-18} & \texttt{2.52E-29} & \texttt{3.49E-33} \\
$2^{12}$ & \texttt{4.55E-13} & \texttt{3.73E-16} & \texttt{1.11E-16} & \texttt{5.16E-18} & \texttt{1.01E-28} & \texttt{1.10E-32} \\
$2^{14}$ & \texttt{1.82E-12} & \texttt{2.19E-17} & \texttt{1.11E-16} & \texttt{1.51E-17} & \texttt{4.04E-28} & \texttt{2.96E-34} \\
$2^{16}$ & \texttt{7.28E-12} & \texttt{1.78E-17} & \texttt{1.11E-16} & \texttt{9.22E-20} & \texttt{1.62E-27} & \texttt{5.41E-34} \\
$2^{18}$ & \texttt{2.91E-11} & \texttt{2.66E-17} & \texttt{1.11E-16} & \texttt{2.99E-19} & \texttt{6.46E-27} & \texttt{2.23E-33} \\
$2^{20}$ & \texttt{1.16E-10} & \texttt{3.58E-17} & \texttt{1.11E-16} & \texttt{1.17E-19} & \texttt{2.58E-26} & \texttt{5.81E-33} \\
\hline
\end{tabular}
\end{table}

\subsection{Dynamical system}
Simulation of dynamical systems is a natural application for compensated summation algorithms, where the solution update can be viewed as a recursive summation with the addends being generated on the fly. 
Compensated summation algorithms can be particularly beneficial for dynamical systems that are sensitive to small perturbations or require long simulation steps such as those arising in celestial mechanics.

In this example, we experiment with a classical three-body problem with a configuration of initial states that lead to periodic orbits resembling the figure-eight shape \cite{chenciner2000remarkable}.
The problem can be formulated as the following ODE system:
\begin{equation}
\label{eqn:3body_figure8}
\left\{ 
\arraycolsep=2pt
\begin{array}{lcr}
\displaystyle \dot{R_i} & = & \displaystyle V_i \, ; \\[0.5ex]
\displaystyle \dot{V_i} & = & \displaystyle A_i \, ,
\end{array}
\right.
\end{equation}
where $i$ ranges from $1$ to $3$ and identifies one of the three interacting bodies; $R_i$, $V_i$, and $A_i$ are vectors with two components each, representing the position, velocity, and acceleration, respectively, of the $i$th body in the orbiting plane.
The acceleration $A_i$ is calculated based on the following formula:
\begin{equation}
A_i = \sum_{j\neq i} \frac{R_j - R_i}{\|R_j - R_i\|^3} \, ,
\end{equation}
where $j$ also goes through $1$ to $3$.

The system \eqref{eqn:3body_figure8} is integrated using the symplectic Euler method as follows:
\begin{equation}
\label{eqn:3body_figure8_semi-implicit-Euler}
\arraycolsep=2pt
\begin{array}{lcl}
R_i^{(k+1)} &=& R_i^{(k)} + h V_i^{(k)}   \, ; \\[0.5ex]
V_i^{(k+1)} &=& V_i^{(k)} + h A_i^{(k+1)} \, ,
\end{array}
\end{equation}
where $k$ indicates the time step and $h$ is the step size. 
The simulation is performed at fp32 with $h = 2^{-11}$ for 10000 periods, with each period being approximately 6.3259.
Results with Algorithm \ref{alg:recursive_summation} (simple recursive summation without compensation) applied to \eqref{eqn:3body_figure8_semi-implicit-Euler} are shown in Figure \ref{fig:3body_figure8_semi-implicit-Euler:10000:0op}, with the plots illustrating the orbit of the first body at around the $0$th, $2000$th, $\cdots$, $8000$th period, respectively.
As can be observed from these plots, the simulation quickly deviates from the figure-eight shaped orbit that the bodies are supposed to follow.

\newcommand{\figurePath}{./code/Euler/plot/fp32/h_11/periods_10000/segments_10/}

\begin{figure}[H]
\centering
\newcount\myloopiter
\myloopiter=0
\loop
    \begin{subfigure}[b]{0.18\textwidth}
        \centering
        \includegraphics[width=\textwidth]{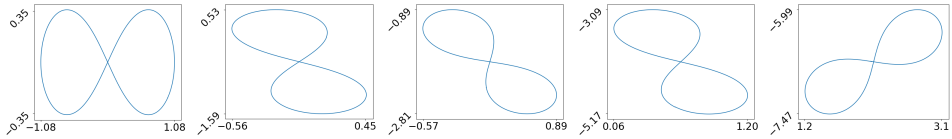}
    \end{subfigure}%
\ifnum\myloopiter<8
    \advance\myloopiter by 2
\repeat
\caption{Simulation of the three-body problem with an initial configuration that should lead to periodic orbits resembling a figure-eight shape. The simulation is performed at fp32 with the simple recursive summation for the solution update. 
Orbit of the first body at around the $0$th, $2000$th, $4000$th, $6000$th, and $8000$th period is shown in the plots, respectively.
The accumulation of round-off errors lead to the deviation from the orbit it should follow as illustrated at the top left plot.}
\label{fig:3body_figure8_semi-implicit-Euler:10000:0op}
\end{figure}

In comparison, results with Algorithm \ref{alg:recursive_summation_6op} (recursive summation with 6op compensation) and Algorithm \ref{alg:recursive_summation_double_6op} (recursive summation with double 6op compensation) applied to \eqref{eqn:3body_figure8_semi-implicit-Euler} are shown in Figures \ref{fig:3body_figure8_semi-implicit-Euler:10000:6op} and \ref{fig:3body_figure8_semi-implicit-Euler:10000:12op}, respectively.
As can be observed from the plots therein, the figure-eight shaped orbit is mostly maintained after 10000 periods of simulation, thanks to the reducation in the accumulation of round-off errors related to the solution update.

\begin{figure}[H]
\centering
\newcount\myloopiter
\myloopiter=0
\loop
    \begin{subfigure}[b]{0.18\textwidth}
        \centering
        \includegraphics[width=\textwidth]{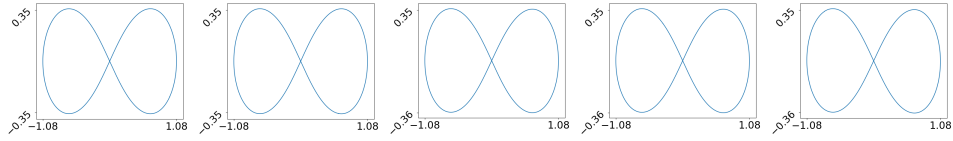}
    \end{subfigure}%
\ifnum\myloopiter<8
    \advance\myloopiter by 2
\repeat
\caption{Same settings as for Figure \ref{fig:3body_figure8_semi-implicit-Euler:10000:0op} except that the 6op compensated recursive summation is used for the solution update.
Thanks to the reduction in the round-off error accumulation, the figure-eight shaped orbit is mostly maintained after 10000 periods of simulation.}
\label{fig:3body_figure8_semi-implicit-Euler:10000:6op}
\end{figure}

\begin{figure}[H]
\centering
\newcount\myloopiter
\myloopiter=0
\loop
    \begin{subfigure}[b]{0.18\textwidth}
        \centering
        \includegraphics[width=\textwidth]{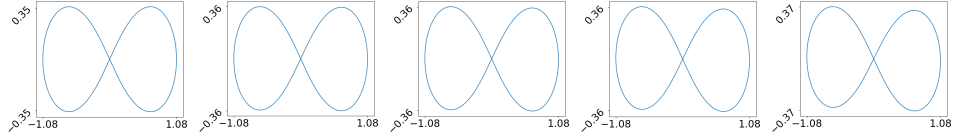}
    \end{subfigure}%
\ifnum\myloopiter<8
    \advance\myloopiter by 2
\repeat
\caption{Same settings as for Figure \ref{fig:3body_figure8_semi-implicit-Euler:10000:0op} except that the double 6op compensated recursive summation is used for the solution update.
Thanks to the reduction in the round-off error accumulation, the figure-eight shaped orbit is mostly maintained after 10000 periods of simulation.}
\label{fig:3body_figure8_semi-implicit-Euler:10000:12op}
\end{figure}

We note here that in the setting presented here, we should not expect the double 6op compensated Algorithm \ref{alg:recursive_summation_double_6op} to perform much better, if any better, than the 6op compensated Algorithm \ref{alg:recursive_summation_6op}.
The reason is evidenced in the error bounds given in \eqref{eqn:recursive_summation_6op_sn+en_error_bound_c} and \eqref{eqn:recursive_summation_double_6op_sn+en_error_bound_c}.
To first-order approximation, the accumulation error is bounded by an $\epsilon$ term and an $n\epsilon^2$ term for Algorithm \ref{alg:recursive_summation_6op}, whereas it is bounded by an $n\epsilon^2$ term for Algorithm \ref{alg:recursive_summation_double_6op}.
The absence of the $\epsilon$ term is what makes Algorithm \ref{alg:recursive_summation_double_6op} superior when the accumulation operation is considered by itself.
However, once another error on the order of $\epsilon$ is introduced in the broader context, the two algorithms will be on an equal footing. 
In the case of ODE simulations, this other error comes from the evaluation and assignment of the new addends, i.e., $h V_i^{(k)}$ and $h A_i^{(k)}$ in \eqref{eqn:3body_figure8_semi-implicit-Euler}, which is on the order of $\epsilon$.
The derived error bounds, along with the numerical experiments, indicate that the error from the accumulation operation is no longer the bottleneck and further improvement for the overall simulation can only be achieved if other error sources are addressed.

Additionally, even with the compensation, the simulation will eventually drift away from the figure-eight shaped orbit, as can be observed in Figures \ref{fig:3body_figure8_semi-implicit-Euler:50000:6op} and \ref{fig:3body_figure8_semi-implicit-Euler:50000:12op} as the correspondences of \ref{fig:3body_figure8_semi-implicit-Euler:10000:6op} and \ref{fig:3body_figure8_semi-implicit-Euler:10000:12op}, respectively, but for a longer simulation time of 50000 periods.

\renewcommand{\figurePath}{./code/Euler/plot/fp32/h_11/periods_50000/segments_10/}

\begin{figure}[H]
\centering
\newcount\myloopiter
\myloopiter=0
\loop
    \begin{subfigure}[b]{0.18\textwidth}
        \centering
        \includegraphics[width=\textwidth]{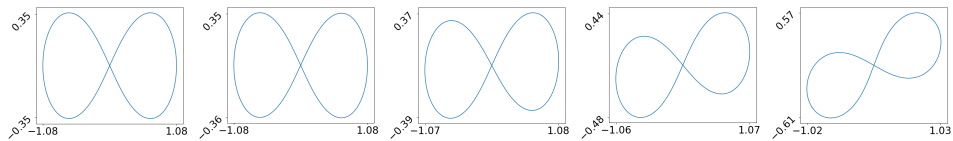}
    \end{subfigure}%
\ifnum\myloopiter<8
    \advance\myloopiter by 2
\repeat
\caption{Same setting as for Figure \ref{fig:3body_figure8_semi-implicit-Euler:10000:6op} except that the simulation is performed for a longer time of 50000 periods with a 10000-period increment from plot to plot. The simulated orbit eventually drifts away from the figure-eight shape even with the compensation.}
\label{fig:3body_figure8_semi-implicit-Euler:50000:6op}
\end{figure}

\begin{figure}[H]
\centering
\newcount\myloopiter
\myloopiter=0
\loop
    \begin{subfigure}[b]{0.18\textwidth}
        \centering
        \includegraphics[width=\textwidth]{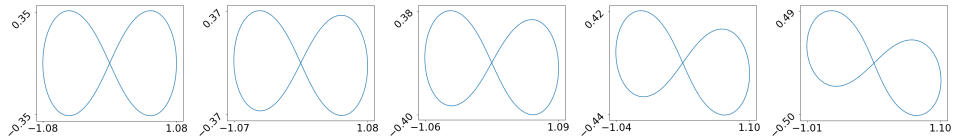}
    \end{subfigure}%
\ifnum\myloopiter<8
    \advance\myloopiter by 2
\repeat
\caption{Same setting as for Figure \ref{fig:3body_figure8_semi-implicit-Euler:10000:12op} except that the simulation is performed for a longer time of 50000 periods with a 10000-period increment from plot to plot. The simulated orbit eventually drifts away from the figure-eight shape even with the compensation.}
\label{fig:3body_figure8_semi-implicit-Euler:50000:12op}
\end{figure}

\section{Conclusions}
\label{sec:conclusions}

Analysis of various compensated recursive summation algorithms is provided using the error free transformation results derived using Dekker's floating point number system. 
Our analysis differs from 
previous efforts in that it provides a complete description of the error behavior instead of tracking the leading order term only. 
Recursive relations that reveal the error propagation pattern are provided as part of the analysis, which can be used to identify the operation that limits the accuracy of the overall algorithm.
This information is used to design more nuanced compensated summation algorithms that remove the accuracy bottleneck.
Numerical evidence is provided to illustrate the efficacy of these algorithms, verifying empirically that in some cases, algorithmic improvement can achieve an accuracy as if a datatype with twice the amount of mantissa bits has been used.

\section{Acknowledgements}
This research used resources of the Argonne Leadership Computing Facility, a U.S. Department of Energy (DOE) Office of Science user facility at Argonne National Laboratory and is based on research supported by the U.S. DOE Office of Science-Advanced Scientific Computing Research Program, under Contract No. DE-AC02-06CH11357.

\bibliographystyle{./bst_base/abbrv} 
\bibliography{refs} 

\end{document}